\documentclass[11pt]{amsart}

\usepackage{graphicx}
\usepackage{latexsym}    
\usepackage{amssymb}   
\usepackage{amsmath}    
\usepackage{amsbsy}
\usepackage{amsthm}
\usepackage{amsgen}
\usepackage{amsfonts}
\usepackage{array}
\usepackage[all]{xy}    
\usepackage{tikz}
\usetikzlibrary{decorations.pathmorphing}
\usepackage{hyperref}
\usepackage{color}
\usepackage{verbatim}
\usepackage{url}
\usepackage{enumerate}
\usepackage{enumitem}

\tikzset{snake it/.style={decorate, decoration=snake}}

\usepackage[normalem]{ulem}

\newtheorem{thm}{Theorem}[section]
\newtheorem{cor}[thm]{Corollary}
\newtheorem{lem}[thm]{Lemma}
\newtheorem{prop}[thm]{Proposition}

\newtheorem{lemma}[thm]{Lemma}

\newtheorem{theorem}{Theorem}

\newtheorem{corollary}[theorem]{Corollary}

\theoremstyle{definition}
\newtheorem{defn}[thm]{Definition}

\newtheorem*{conj*}{Conjecture}
\newtheorem{example}[thm]{Example}
\newtheorem{rem}[thm]{Remark}

\newtheorem*{defn*}{Definition}

\newtheorem*{ack}{Acknowledgements}

\numberwithin{equation}{section}

\newcommand{\RR}{\mathbb{R}}
\newcommand{\CC}{\mathbb{C}}

\newcommand{\reg}{\textrm{reg}}

\DeclareMathOperator{\diag}{diag}

\DeclareMathOperator{\rk}{rk}
\DeclareMathOperator{\Ker}{Ker}

\DeclareMathOperator{\Span}{span}

\newcommand{\horb}{\hat{\mathcal{O}}}
\newcommand{\hP}{\hat{P}}
\newcommand{\hL}{\hat{L}}
\newcommand{\hM}{\hat{M}}

\newcommand{\CP}{\mathbb{C}P}

\newcommand{\sph}{\mathrm{\mathbb{S}}}

\newcommand{\ZZ}{\mathbb{Z}}

\newcommand{\SO}{\mathrm{SO}}
\newcommand{\SU}{\mathrm{SU}}
\newcommand{\U}{\mathrm{U}}

\newcommand{\Hom}{\mathrm{Hom}}

\newcommand{\orb}{\mathcal{O}}
\newcommand{\odd}{\mathrm{odd}}
\newcommand{\Q}{\mathbb{Q}}

\def\ol{\overline}
\def\ul{\underline}

\def\x{\times}

\def\lra{\longrightarrow}
\def\lmt{\longmapsto}

\def\In{\subseteq}
\def\Ni{\supseteq}

\def\CC{\mathbb{C}}

\def\QQ{\mathbb{Q}}
\def\RR{\mathbb{R}}
\def\ZZ{\mathbb{Z}}

\def\mc{\mathcal}
\def\mf{\mathfrak}

\def\vphi{\varphi}

\def\g{\mf{g}}

\def\id{\mathrm{id}}

\begin{document}



\title[Torus orbifolds]{Torus orbifolds, slice-maximal torus actions and rational ellipticity}



\author[F.~Galaz-Garc\'ia]{F.~Galaz-Garc\'ia$^*$}
\address[Galaz-Garc\'ia]{
Institut f\"ur Algebra und Geometrie, Karlsruher Institut f\"ur Technologie (KIT), Germany.}
\email{galazgarcia@kit.edu}
\thanks{$^*$ Received support from SFB 878: \emph{Groups, Geometry \& Actions} at WWU M\"unster.}


\author[M.~Kerin]{M.~Kerin$^*$}
\address[Kerin]{Mathematisches Institut, WWU M\"unster, Germany.}
\email{m.kerin@math.uni-muenster.de}


\author[M.~Radeschi]{M.~Radeschi$^*$}
\address[Radeschi]{Mathematisches Institut, WWU M\"unster, Germany.}
\email{mrade\_02@uni-muenster.de}
\thanks{}


\author[M.~Wiemeler]{M.~Wiemeler$^\dagger$}
\address[Wiemeler]{Institut f\"ur Mathematik, Universit\"at Augsburg, Germany.}
\email{michael.wiemeler@math.uni-augsburg.de}
\thanks{$\dagger$ Partially supported by DFG Grant HA 3160/6-1.}

\date{\today}


\subjclass[2010]{55P62,57S10}
\keywords{torus, orbifold, rationally, elliptic, non-negative curvature}


\begin{abstract}
In this work, it is shown that a simply-connected, rationally-elliptic torus orbifold is equivariantly rationally homotopy equivalent to the quotient of a product of spheres by an almost-free, linear torus action, where this torus has rank equal to the number of odd-dimensional spherical factors in the product.  As an application, simply-connected, rationally-elliptic manifolds admitting slice-maximal torus actions are classified up to equivariant rational homotopy.  The case where the rational-ellipticity hypothesis  is replaced by non-negative curvature is also discussed, and the Bott Conjecture in the presence of a slice-maximal torus action is proved.
 \end{abstract}

\maketitle




\section{Introduction}


A \emph{torus manifold} is a  $2n$-dimensional, closed, orientable, smooth manifold equipped with a smooth, effective $n$-torus action which has non-empty fixed-point set. Such spaces have been of long-standing interest, going back, on the one hand, to Orlik and Raymond's work on closed, smooth $4$-manifolds equipped with smooth, effective $T^2$ actions \cite{OR1,OR2} and, on the other hand, to the study of toric varieties in algebraic geometry \cite{Fu}. Many results on  manifolds with torus actions admit generalizations  to orbifolds (see, for example, \cite{HaSa} for smooth torus actions on orbifolds,  \cite{HoMa,LeTo} for Hamiltonian torus actions on symplectic orbifolds or \cite{PoSa} for quasitoric orbifolds).

Recently, it has been shown in \cite{Wi} that, if $M$ is simply connected and either a rationally-elliptic torus manifold with torsion-free integer cohomology or a torus manifold with non-negative sectional curvature, then $M$ is homeomorphic to the quotient of a product of spheres by a free, linear torus action. In this paper, \emph{torus orbifolds} are investigated, and a similar result to that in \cite{Wi} is proven.

Recall that a simply-connected topological space $X$ is called \emph{rationally elliptic} if it satisfies $\dim_\Q H^*(X;\Q) < \infty$ and $\dim_\Q (\pi_*(X)\otimes \Q) < \infty$.  Two spaces $X$ and $Y$ are \emph{rationally homotopy equivalent} if their corresponding minimal models are isomorphic. Given a torus $T$,
a rational homotopy equivalence between $T$-spaces $X$ and $Y$ is \emph{$T$-equivariant} if the corresponding Borel constructions $X_T$ and $Y_T$ are also rationally homotopy equivalent and there exists a commutative diagram
\[
\xymatrix{
H^*(Y;\QQ) \ar[r]& H^*(X;\QQ) \\
H^*_{T}(Y;\QQ) \ar[u]^{} \ar[r]&  H^*_{T}(X;\QQ)\ar[u]^{}
}
\]
where the horizontal arrows are isomorphisms induced by the respective rational homotopy equivalences.


\begin{theorem}
\label{T:main_thm} 
Let $(\orb,T)$ be a rationally-elliptic, simply-connected torus orbifold.
Then there is a product $\hP$ of spheres of dimension $\geq 3$, a torus $\hat{L}$ acting linearly and almost freely on $\hP$, and an effective, linear action of $T$ on  $\hat{\orb}=\hP/\hat{L}$, such that there is a $T$-equivariant rational homotopy equivalence $\orb\simeq_{\QQ}\hat{\orb}$.

Moreover, if $\orb$ is a manifold, then $\hat{L}$ acts freely on $\hP$ and thus $\hat{\orb}$ is a manifold as well.
\end{theorem}

The final statement in Theorem \ref{T:main_thm} regarding manifolds is closely related to Theorem 1.1 of \cite{Wi}, where a stronger assumption (torsion-free integral cohomology) is required in order to obtain a correspondingly strong conclusion (classification up to homeomorphism).

Torus orbifolds have been studied in \cite{Ha,HaMa1,HaMa2} and arise naturally in the study of smooth torus actions on manifolds, for example, when the action is slice maximal.


\begin{defn*}
Let $M$ be a closed, orientable, smooth $n$-manifold on which a torus $T^k$ acts smoothly and effectively, and let $m$ be the minimal dimension of an orbit. The action is \emph{slice maximal} if $2k=n+m$.
 \end{defn*}

It is clear from the definition that torus manifolds are an extremal case of slice-maximal actions.  For a generic $k$-torus action on an $n$-dimensional manifold, it follows from the slice representation at a minimal orbit that $2k\leq n+m$. Thus, if equality holds, the slice representation at a minimal orbit is even dimensional and has maximal symmetry rank, justifying the terminology  ``slice maximal''.  Slice-maximal actions were considered in \cite{Is, Us}, where they were called \emph{maximal}.

Given an $n$-manifold $M$ with a slice-maximal $T^k$ action, there exists a subtorus $T^m \In T^k$ acting almost freely on $M$ and the quotient $\orb = M/T^m$ is a $2(k-m)$-dimensional torus orbifold.  Moreover, if $M$ is rationally elliptic, so too is the quotient $\orb$.

By applying Theorem \ref{T:main_thm}, it turns out that the existence of a slice-maximal torus action has strong implications on the topology of a manifold.


\begin{theorem} 
\label{T:CORE_THM}
Let $M$ be an $n$-dimensional, smooth, closed, simply-connected, rationally-elliptic manifold with a slice-maximal $T^k$ action.  Then there is a product $\hP$ of spheres of dimension $\geq 3$, a torus $\hat{K}$ acting linearly and freely on $\hP$, and an effective, linear action of $T^k$ on  $\hM=\hP/\hat{K}$, such that there is a $T^k$-equivariant rational homotopy equivalence $M\simeq_{\QQ}\hM$.
\end{theorem}

As a first application, Theorem \ref{T:CORE_THM} has been used in \cite{GKR} to obtain a classification of closed, simply-connected, rationally-elliptic manifolds admitting effective torus actions of maximal rank up to equivariant rational homotopy equivalence.


For another interesting consequence of Theorem \ref{T:CORE_THM}, recall that the largest integer $r$ for which a closed, simply-connected space $M$ admits an almost-free $T^r$-action is called the \emph{toral rank} of $M$, and is denoted $\rk(M)$.  The Toral Rank Conjecture, formulated by S.\ Halperin, asserts that $\dim H^*(M ; \Q) \geq 2^{\rk(M)}$.


\begin{corollary} 
Let $M$ be a smooth, closed, simply-connected, rationally-elliptic, $n$-dimensional manifold with a slice-maximal torus action.  Then $M$ satisfies the Toral Rank Conjecture.
\end{corollary}


\begin{proof} Let $T^r$ act almost freely on $M$.  Given $H^2(M;\Q) = \Q^{b_2(M)}$, there is a principal $T^{b_2(M)}$-bundle over $M$ with (rationally) $2$-connected total space $P$.  As any action by a torus $T$ on $M$ lifts to a $T \x T^{b_2(M)}$ action on $P$, the slice-maximal action (resp. the almost-free $T^r$ action) on $M$ lifts to a slice-maximal action (resp. an almost-free $T^r \x T^{b_2(M)}$ action) on $P$.  By Theorem \ref{T:CORE_THM} and since $H^2(P;\Q) = 0$, $P$ must have the rational cohomology of a product of spheres of dimension $\geq 3$.  By  \cite[Prop.\ 7.23]{FOT}, $P$ satisfies the Toral Rank Conjecture, i.e. $H^*(P;\Q) \geq 2^{r + b_2(M)}$.  The result now follows from $\dim H^*(P ; \Q) \leq \dim H^*(T^{b_2(M)} ; \Q) \cdot \dim H^*(M ; \Q)$.
\end{proof}

Finally, recall that the Bott Conjecture asserts that a closed, simply-connected, non-negatively-curved Riemannian manifold is rationally elliptic.  In \cite{Sp}, W.\ Spindeler verified the conjecture for simply-connected, non-negatively-curved torus manifolds.  In fact, the conjecture also holds in the slice-maximal setting.


\begin{theorem}
\label{T:Bott}
Let $M$ be a closed, simply-connected, non-negatively-curved Riemannian manifold admitting an isometric, slice-maximal torus action.  Then $M$ is rationally elliptic.
\end{theorem}

It is worth noting that non-negatively-curved torus manifolds have already been classified up to equivariant diffeomorphism in \cite{Wi}.  Furthermore, if the non-negative-curvature hypothesis in Theorem \ref{T:Bott} were to be replaced by positive curvature, then it would follow from the work of K.\ Grove and C.\ Searle \cite{GS} that $M$ is equivariantly diffeomorphic to a sphere or complex projective space equipped with a linear action.
\medskip


The paper is organized as follows: In Section~\ref{S:Setup}, some basic definitions and facts about orbifolds are collected, following the presentation in \cite{KL}, as well as some results on smooth actions on orbifolds. 
These results have  been included to provide a basic reference for compact Lie group actions on orbifolds, since they seem to be scattered in the literature  (see, for example, ~\cite{HaSa90,HaSa,LeTo,Ye}).
In Section~\ref{S:TORBIFOLDS}, torus orbifolds are introduced and their fundamental properties established. In Section~\ref{SS:MANGLECPX}, there is a brief review of GKM-theory applied to torus orbifolds. The proof of Theorem~\ref{T:main_thm} is contained in Section~\ref{S:PROOF}. In Section~\ref{S:Example}, an example of a family of rationally elliptic torus orbifolds which are not rationally homotopy equivalent to any rationally elliptic manifold is provided, illustrating that almost-free (rather than free) torus actions on products of spheres are necessary in the conclusion of the Theorem~\ref{T:main_thm}.  Section~\ref{S:CoreThm} is devoted to establishing Theorem~\ref{T:CORE_THM}.  In Section~\ref{S:Nonneg}, a version of Theorem~\ref{T:main_thm} for non-negatively-curved orbifolds of dimension $\leq 6$ is proven.  The case of general dimensions remains open.  Section~\ref{S:Nonneg} concludes with the proof of Theorem \ref{T:Bott}, which is independent of the rest of the paper.

The reader is referred to \cite{FHT}  for the basic definitions and results of rational homotopy theory.


\begin{ack} We would like to thank Matthias Franz for helpful comments on a previous version of this article, and Wolfgang Spindeler for discussions regarding his work in \cite{Sp}.
\end{ack}


\section{Review of orbifolds}
\label{S:Setup}


\begin{defn}
\label{D:local_model}
A \emph{local model of dimension $n$} is a pair $(\tilde{U}, \Gamma)$, where $\tilde{U}$ is an open, connected subset of a Euclidean space $\RR^n$, and $\Gamma$ is a finite group acting smoothly and effectively on $\tilde{U}$. 

A \emph{smooth map} $(\tilde{U}_1,\Gamma_1)\to (\tilde{U}_2,\Gamma_2)$ between local models $(\tilde{U}_i,\Gamma_i)$, $i=1,2$, is a homomorphism $\varphi_{_\#}:\Gamma_1\to \Gamma_2$ together with a $\varphi_{_\#}$-equivariant smooth map $\tilde{\varphi}:\tilde{U}_1\to \tilde{U}_2$, i.e. $\tilde\vphi(\gamma\cdot \tilde u) = \vphi_\#(\gamma) \cdot \tilde\vphi(\tilde u)$, for all $\gamma \in \Gamma_1$, $\tilde u \in \tilde U_1$.
\end{defn}

Given a local model $(\tilde{U},\Gamma)$, denote by $U$ the quotient $\tilde{U}/\Gamma$. Clearly, a smooth map $\tilde{\varphi}:(\tilde{U}_1,\Gamma_1)\to (\tilde{U}_2,\Gamma_2)$ induces a map $\varphi:U_1\to U_2$.
The map $\varphi$ is called an \emph{embedding} if $\tilde{\varphi}$ is an embedding. In this case, the effectiveness of the actions in the local models implies that $\varphi_{_\#}$ is injective.


\begin{defn}
An \emph{$n$-dimensional local chart} $(U_p, \tilde U_p, \Gamma_p, \pi_p)$ around a point $p$ in a topological space $X$ consists of:
\begin{enumerate}
\item A neighbourhood $U_p$ of $p$ in $X$;
\item A local model $(\tilde{U}_p, \Gamma_p)$ of dimension $n$;
\item A $\Gamma_p$-equivariant projection $\pi_p:\tilde{U}_p\to U_p$, where $\Gamma_p$ acts trivially on $U_p$, that induces a homeomorphism $\tilde{U}_p/\Gamma_p\to U_p$.
\end{enumerate}
If $\pi_p^{-1}(p)$ consists of a single point, $\tilde{p}$, then $(U_p, \tilde U_p, \Gamma_p, \pi_p)$ is called a \emph{good local chart} around $p$.  In particular, $\tilde p$ is fixed by the action of $\Gamma_p$ on $\tilde{U}_p$.
\end{defn}

Note that, given a good local chart $(U_p, \tilde U_p, \Gamma_p, \pi_p)$ around a point $p$ in a topological space $X$, the $4$-tuple $(U_p, \tilde U_p, \Gamma_p, \pi_p)$ is also a local chart, not necessarily good, around any other point $q \in U_p$.  By abusing notation, a local chart $(U, \tilde U, \Gamma, \pi)$ will from now on be denoted simply by $U$.

\begin{defn}
An \emph{$n$-dimensional orbifold atlas} for a topological space $X$ is a collection of $n$-dimensional local charts $\mc{A}=\{U_{\alpha}\}_\alpha$ such that the neighbourhoods $U_\alpha \in \mc{A}$ give an open covering of $X$ and:
\newenvironment{Disp}
{\begin{list}{}{%
    \setlength{\leftmargin}{4mm}}
  \item[] \ignorespaces}
{\unskip \end{list}}
\begin{Disp}
For any $p\in U_{\alpha}\cap U_{\beta}$, there is a local chart $U_\gamma \in \mc{A}$ with $p\in U_{\gamma}\In U_{\alpha}\cap U_{\beta}$ and embeddings $(\tilde{U}_{\gamma}, \Gamma_{\gamma})\to(\tilde{U}_{\alpha}, \Gamma_{\alpha})$, $(\tilde{U}_{\gamma}, \Gamma_{\gamma})\to (\tilde{U}_{\beta}, \Gamma_{\beta})$.
\end{Disp}
Two $n$-dimensional atlases are called \emph{equivalent} if they are contained in a third atlas. 
\end{defn}


\begin{defn}
An \emph{$n$-dimensional (smooth) orbifold}, denoted by $\orb^n$ or simply $\orb$, is a second-countable, Hausdorff topological space $|\orb|$, called the \emph{underlying topological space} of $\orb$, together with an \emph{equivalence class of $n$-dimensional orbifold atlases}. 

An orbifold is \emph{orientable} if every local model $\tilde{U}_{\alpha}$ is orientable, and if every $\Gamma_{\alpha}$ action and every embedding $\tilde{U}_{\gamma}\to \tilde{U}_{\alpha}$ is orientation preserving.   
Given an orientable orbifold $\orb$, it is not hard to see that the set of points $p\in \orb$ for which $\Gamma_p$ is non-trivial has codimension at least $2$ in $\orb$.  An orbifold $\orb$ is \emph{connected} (resp. \emph{closed}) if its underlying topological space $|\orb|$ is connected (resp.\ compact and without boundary).
\end{defn}

Given an orbifold $\orb$ and any point $p\in \orb$, one can always find a good local chart $U_p$ around $p$.  Moreover, the corresponding group $\Gamma_p$ does not depend on the choice of good local chart around $p$, and is referred to as the \emph{local group at $p$}.  From now on, only good local charts will be considered.

\begin{lem}
\label{L:localgp}
Let $\orb$ be an orbifold and $U_p$ a good local chart around $p \in \orb$.  Let $q \in U_p$, $\tilde q \in \pi_p^{-1}(q) \In U_p$ and $(\Gamma_p)_{\tilde q} = \{\gamma \in \Gamma_p \mid \gamma \cdot \tilde q = \tilde q \}$.  Then there exists a $(\Gamma_p)_{\tilde q}$-invariant neighbourhood $\tilde U_q \In \tilde U_p$ of $\tilde q$ such that $(\pi_p(\tilde U_q), \tilde U_q, (\Gamma_p)_{\tilde q}, \pi_p|_{\tilde U_p})$ is a good local chart around $q$.
\end{lem}

\begin{proof}
Define $\tilde U_q$ to be a $(\Gamma_p)_{\tilde q}$-invariant neighbourhood of $\tilde q$ such that, for every $\gamma \in \Gamma_p \setminus (\Gamma_p)_{\tilde q}$, one has $\tilde U_q \cap \tilde \gamma \cdot \tilde U_q = \emptyset$.  Then, by definition of $\pi_p$, $U_q = \pi_p(\tilde U_q) \In U_p$ is open and the restriction $\pi_p|_{\tilde U_q} : \tilde U_q \to U_q$ is an open, continuous, $(\Gamma_p)_{\tilde q}$-equivariant map.  As the pre-image of each point in $U_q$ is an orbit of the $(\Gamma_p)_{\tilde q}$ action on $\tilde U_q$, $\pi_p|_{\tilde U_q}$ induces a homeomorphism $\tilde U_q / (\Gamma_p)_{\tilde q} \to U_q$.  Hence, $U_q$ is a good local chart around $q$.
\end{proof}

In particular, given a good local chart $U_p$ around $p \in \orb$, a point $q \in U_p$ and $\tilde q \in \pi_p^{-1}(q)$, one can identify the local group $\Gamma_q$ at $q$ with $(\Gamma_p)_{\tilde q}$.


\begin{defn}
A \emph{smooth map} $\varphi: \orb_1 \to  \orb_2$ between orbifolds is given by a continuous map $|\varphi| : |\orb_1| \to |\orb_2|$ such that, if $U_p$ and $U_{\varphi(p)}$ are (good) local charts around $p \in \orb_1$ and $\varphi(p) \in \orb_2$, respectively, such that $\varphi\left(U_p\right)\In U_{\varphi(p)}$, then there is a (possibly non-unique) \emph{smooth lift at $p \in \orb_1$}, $\tilde{\varphi}_p : (\tilde{U}_p, \Gamma_p) \to (\tilde{U}_{\varphi(p)}, \Gamma_{\varphi(p)})$, so that the diagram
\[
\xymatrix{\ar @{} [drr] |\circlearrowright
 \tilde{U}_p \ar[d]_{\pi_p} \ar[rr]^{\tilde{\varphi}_p}_{C^\infty} && \tilde{U}_{\varphi(p)} \ar[d]^{\pi_{\varphi(p)}} \\
 U_p \ar[rr]_{\varphi}        && U_{\varphi(p)}        }
\]
commutes and there is an induced homomorphism $(\tilde \vphi_p)_\#:\Gamma_p\to \Gamma_{\varphi(p)}.$

A  \emph{diffeomorphism} $\varphi:\orb_1\to \orb_2$ between orbifolds is a smooth map with a smooth inverse. In this case, $\Gamma_p$ is isomorphic to $\Gamma_{\varphi(p)}$ and, given two (smooth) lifts $\tilde{\varphi}_1$, $\tilde{\varphi}_2$ of $\varphi$, there exists $\gamma\in \Gamma_{\varphi(p)}$ such that $\tilde{\varphi}_1=\gamma\cdot \tilde{\varphi}_2$.  Further properties of smooth maps between orbifolds are discussed in \cite{Bo}.
\end{defn}


\begin{rem} 
For a general smooth map $\varphi: \orb_1 \to  \orb_2$, it is not always the case that any two lifts $\tilde{\varphi}_1$, $\tilde{\varphi}_2$ of $\varphi$ at $p\in \orb_1$ differ by composition with an element in $\Gamma_{\varphi(p)}$. Indeed, the map $f:\mathbb{R}\rightarrow \mathbb{R}/{\pm1}$ given by
\[
   f(x) = \left\{
     \begin{array}{lr}
        0, & x =0; \\
       \exp(-1/x^2), & x\neq 0
     \end{array}
   \right.
\]
is smooth and has four different lifts at $0$. The local group at $x=0$, however, is $\mathbb{Z}_2$.
\end{rem}

 
\begin{defn}
An orbifold $\orb_1$ is a \emph{suborbifold} of an orbifold $\orb_2$, if there is a smooth map $\varphi:\orb_1\to \orb_2$ such that $|\varphi|$ maps $|\orb_1|$ homeomorphically onto its image in $|\orb_2|$ and, for every $p\in \orb_1$, some (and, hence, every) 
smooth lift $\tilde\varphi_p:\tilde{U}_{p}\to\tilde{U}_{\varphi(p)}$ is an immersion. In this case, $\orb_1$ will be identified with its image.
\end{defn}


\begin{defn}
A suborbifold $\orb_1\In \orb_2$ is a \emph{strong suborbifold} if, for every $p\in\orb_1$ and every good local chart $U_p$, the image of a smooth lift $\tilde{\varphi}_p$ is independent of the choice of lift. 
\end{defn}

If $\orb_1$ is a suborbifold of $\orb_2$ and $p\in \orb_1$, the local group $\Gamma_p^1$ of $p$ in $\orb_1$ and the local group $\Gamma_p^2$ of $p$ in $\orb_2$ may, in general, be different, but it is always the case that $\Gamma_p^1\In \Gamma_p^2$. Notice that  when $\orb_1$ is a strong suborbifold and $U_p$ is a good local chart in $\orb_2$, the set $\pi_p^{-1}(U_p\cap \orb_1)$ is a submanifold preserved by the induced $\Gamma^2_p$ action. The above definition of strong suborbifold is equivalent to Thurston's definition of suborbifold (cf.~\cite{Thu}).

Given a strong suborbifold $\orb_1 \In \orb_2$, let $U_p$ be a good local chart (in $\orb_2$) around $p \in \orb_1$ and let $\tilde{T}_p U_p$ be the tangent space to $\tilde{U}_p$ at $\tilde{p}=\pi_p^{-1}(p)$.  Denote by $\tilde{T}_p \orb_1 \In \tilde{T}_p U_p$ the tangent space to $\pi_p^{-1}(\orb_1 \cap U_p)$ at $\tilde{p}$.  
Then the space $\tilde{T}_pU_p$ splits as $\tilde{T}_p \orb_1\oplus \tilde{\nu}_p \orb_1$, where $\tilde{\nu}_p \orb_1$ denotes the normal space to $\tilde{T}_p \orb_1 \In \tilde{T}_p U_p$.


\subsection*{Smooth actions on orbifolds}\label{SS:smooth-actions}


\begin{defn}
Let $\orb$ be an orbifold. A \emph{smooth action} of a Lie group $G$ on $\orb$ is a smooth map
\begin{align*}
\varphi:G\times \orb &\rightarrow \orb\\
(g,p) &\mapsto g\cdot p
\end{align*}
such that the following conditions hold:
\begin{itemize}
	\item $g_1\cdot (g_2\cdot p)=(g_1 g_2)\cdot p$ for any $g_1,g_2\in G$ and $p\in \orb$ and
	\item $e\cdot p = p$ for any $p\in \orb$.
\end{itemize}
\end{defn}

The set $G(p)=\{g\cdot p\ |\ g\in G\}$ is the \emph{orbit} of $G$ through $p \in \orb$.   The \emph{ineffective kernel} of the action  is the normal subgroup $\Ker= \{g \in G \mid \vphi(g,\cdot) = \id_\orb\}$.   If the ineffective kernel is trivial, the action is \emph{effective}.  The group $G/\Ker$ will always act effectively.  The \emph{isotropy subgroup} $G_p$ at $p \in \orb$ is the subgroup consisting of those elements in $G$ that fix $p$.  Note that, whenever $G$ is compact, one can always find a $G_p$-invariant good local chart around $p$.
If $G_p$ is trivial (resp.\ finite) for every $p\in \orb$, the action is \emph{free} (resp.\ \emph{almost free}).  The orbit space of the action will be denoted by $\orb/G$ and the fixed-point set $\{p \in \orb \mid G_p = G\}$ by $\orb^G$.  The identity component of $G$ is denoted by $G^o$.  


\begin{lem}
\label{L:strong}
Every $G$-orbit in $\orb$ is a manifold, as well as a strong suborbifold of $\orb$.
\end{lem}  


\begin{proof}
The fact that \(G(p)\), $p \in \orb$, is a manifold and a suborbifold follows as in the manifold case. To see that it is a strong suborbifold, one can argue as follows: Since \(G\) acts by diffeomorphisms, the local groups at all points in the orbit are isomorphic.  However, by Lemma \ref{L:localgp}, the local group of a point \(q\) in a good local chart \(U_p\) around \(p\) is given by the isotropy subgroup of the $\Gamma_p$ action on $\tilde U_p$ at a point \(\tilde{q}\in \pi_p^{-1}(q)\). Therefore, \(\pi_p^{-1}(U_p\cap G(p))\) is contained in the fixed-point set of \(\Gamma_p\) in \(\tilde{U}_p\) and, hence, the restriction of \(\pi_p\) to \(\pi_p^{-1}(U_p\cap G(p))\) is a homeomorphism.  Consequently, the lift \(\tilde{\varphi}_p\) of the inclusion \(\varphi\) of a chart of \(G(p)\) into \(U_p\) is given by \(\pi_p^{-1}\circ \varphi\).  Therefore \(G(p)\) is a strong suborbifold.
\end{proof}
 

\begin{prop}
\label{P:LES} Let $\orb$ be  an orbifold with a smooth, effective action by a compact Lie group $G$.  Let $p\in \orb$ have isotropy subgroup $G_p\In G$ and let $U_p$ be a $G_p$-invariant good local chart. Then there exists a Lie group $\tilde{G}_p$ such that:
\begin{enumerate}
	\item $\tilde{G}_p$ acts on $\tilde{U}_p$ and $\tilde{U}_p/\tilde{G}_p = U_p/G_p$;
	\item $\tilde{G}_p$ is an extension of $G_p$ by $\Gamma_p$, i.e.~there exists a short exact sequence
	\[
	\{\,e\,\}\rightarrow \Gamma_p \rightarrow \tilde{G}_p \stackrel{\rho}{\rightarrow} G_p \rightarrow \{\,e\,\}.
	\]
\end{enumerate}
\end{prop}


\begin{proof}
Let $g\in G_p$. The action of $G_p$ on $U_p$ gives a smooth map 
\begin{align*}
L_g:U_p&\rightarrow U_p\\
q&\mapsto g \cdot q.
\end{align*}
Then, by definition, there exists a smooth lift $\tilde{L}_g:\tilde{U}_p\rightarrow \tilde{U}_p$ of $L_g$. Let $\tilde{G}_p=\{\, F_g : \tilde{U}_p\rightarrow \tilde{U}_p \mid \pi_p\circ F_g=L_g\circ \pi_p,\ g\in G_p\,\}$ be the collection of all possible lifts. This  is a group and, given that the $G_p$ action is smooth, it is not difficult to see that $\tilde{G}_p$ is a Lie group acting smoothly and effectively on $\tilde U_p$. Note that, since $\Gamma_p=\{\,F_e : \tilde{U}_p\rightarrow \tilde{U}_p \mid \pi_p\circ F_e=e\circ \pi_p=\pi_p\}$, $\Gamma_p$ is a normal subgroup of $\tilde{G}_p$. Moreover, $\Gamma_p$ acts on $\tilde{G}_p$ via $F_g\mapsto \gamma\cdot F_g$ and $\tilde{G}_p/\Gamma_p=G_p$, i.e.~the quotient by $\Gamma_p$ fixes a choice of lift corresponding to $L_g$. It then follows that $\tilde{U}_p/\tilde{G}_p = U_p/G_p$.
\end{proof}


\begin{cor}
\label{C:commute}
Let $\orb$ be  an orbifold with a smooth, effective action by a compact Lie group $G$. Let $p\in \orb$ have isotropy subgroup $G_p$ and let $U_p$ be a $G_p$-invariant good local chart. Then  the local group $\Gamma_p$ commutes with every connected subgroup of the lift $\tilde{G}_p$.
\end{cor}


\begin{proof} Let $\tilde{H}$ be a connected subgroup of $\tilde G_p$, $\tilde{g}\in \tilde{H}$, and $\gamma\in \Gamma_p$. From the short exact sequence in Proposition \ref{P:LES}, the element $\tilde{g}\gamma\tilde{g}^{-1}$ belongs to $\Gamma_p$. Since $\tilde{H}$ is connected and $\Gamma_p$ is discrete, the map $\tilde{g}\mapsto \tilde{g}\gamma\tilde{g}^{-1}$ must be constant and hence $\tilde{g}\gamma\tilde{g}^{-1}=\gamma$ for all $\tilde{g}\in \tilde{H}$. 
\end{proof}


\begin{cor}\label{C:fixed-pt}
Let $G$ be a compact, connected Lie group acting smoothly and effectively on  an orbifold $\orb$ such that the fixed-point set $\orb^G$ is non-empty. Then each connected component of $\orb^G$ is a strong suborbifold.
\end{cor}


\begin{proof}
Let $p\in \orb^G$ and let $U_p$ be a $G$-invariant good local chart around $p$. The goal is to prove that $\pi_p^{-1}(\orb^G\cap U_p)$ is a submanifold in $\tilde{U}_p$.

Since $G$ is connected, the map $\tilde{G}^o\to G$ is a covering and therefore for every $g\in G$ there is a $\tilde{g}\in \tilde{G}^o$ projecting to $g$. Let $q\in \orb^G\cap U_p$ and choose $\tilde{q}\in \pi^{-1}(q) \In \tilde U_p$. As $g\cdot q=q$, for every $g\in G$, it follows that for every $\tilde{g}\in \tilde{G}^o$ there is a $\gamma_{\tilde{g}}\in \Gamma_p$ such that $\tilde{g}\cdot \tilde{q}=\gamma_{\tilde{g}}\cdot\tilde{q}$. But $\tilde{G}^o$ is connected, hence $\gamma_{\tilde{g}}=e$ for every $\tilde{g}\in \tilde{G}^o$. Thus $\tilde{q}$ is fixed by $\tilde{G}^o$ and so $\pi_p^{-1}(\orb^G \cap U_p)\In \tilde{U}_p^{\tilde{G}^o}$. The other inclusion trivially holds, and therefore
\[
\pi_p^{-1}(\orb^G \cap U_p)= \tilde{U}_p^{\tilde{G}^o}.
\]

By Corollary \ref{C:commute}, $\Gamma_p$ commutes with $\tilde{G}^o$ and, in particular, $\Gamma_p$ preserves the fixed-point set of $\tilde{G}^o$. Then $\orb^G \cap U_p$ is a strong suborbifold and, since $p$ was arbitrary, it follows that $\orb^G$ is a strong suborbifold.
\end{proof}


\begin{rem}
If a non-connected Lie group $G$ acts smoothly on an orbifold $\orb$, the fixed-point set $\orb^{G}$ need not even be a suborbifold. For example, let $\orb=\RR^2/\Gamma$, where $\Gamma=\ZZ_2$ is generated by $(x,y)\mapsto (-x,-y)$. The orbifold $\orb$ admits a global model $\pi:\RR^2\to \RR^2/\Gamma$. When $G=\ZZ_2$ acts on $\orb$ by the action $[x,y]\mapsto [x,-y]$, the pre-image of $\orb^G$ in $\RR^2$ is given by $\{x=0\}\cup \{y=0\}$, and thus $\orb^{G}$ is not an orbifold.
\end{rem}

Just as for manifolds, one has a notion of Riemannian metric for orbifolds.  An \emph{orbifold-Riemannian metric}  is given at each point $p$ in the orbifold by the metric on a good local chart $U_p$ around $p$ induced by a $\Gamma_p$-invariant Riemannian metric on $\tilde U_p$.
An orbifold equipped with an orbifold-Riemannian metric will be referred to as a \emph{Riemannian orbifold}.
 It is clear that Riemannian notions such as geodesics and completeness carry over to Riemannian orbifolds.
Recall that any orbifold on which a compact Lie group $G$ acts smoothly and effectively admits a $G$-invariant orbifold-Riemannian metric.  Kleiner's Isotropy Lemma (cf.~\cite{Kl}) also holds, with the same proof, in the context of complete Riemannian orbifolds.


\begin{lem}[Isotropy Lemma]
\label{L:Isot_Lem} Let $\orb$ be a  complete Riemannian orbifold and suppose that a compact Lie group $G$ acts effectively and isometrically on $\orb$. Let $c:[0,d]\rightarrow \orb$ be a minimal geodesic between the orbits $G(c(0))$ and $G(c(d))$. Then, for any $t\in (0,d)$, $G_{c(t)}=G_{c}$ does not depend on \(t\) and is a subgroup of $G_{c(0)}$ and of $G_{c(d)}$.
\end{lem}


It turns out that the good local charts given by Lemma \ref{L:localgp} can be chosen to be compatible with the action of a Lie group.

\begin{lem}
\label{L:Isotropy dimension}
Let a compact Lie group $G$ act effectively and isometrically on a complete Riemannian orbifold $\orb$ and fix $p\in \orb$.  Then there exists a $G_p$-invariant good local chart $U_p$ around $p$ such that, for every $q \in U_p$ and every $\tilde q \in \pi_p^{-1}(q)$, there is a $G_q$-invariant good local chart $U_q \In U_p$ around $q$ and a commutative diagram
\begin{equation}\label{E:A}
\xymatrix{
\{e\}  \ar[r] & \Gamma_q \ar[r] \ar[d] & \tilde G_q \ar[r]^{\rho_q} \ar[d] & G_q \ar[r] \ar[d] & \{e\} \\
\{e\} \ar[r] & \Gamma_p \ar[r] & \tilde G_p \ar[r]^{\rho_p} & G_p \ar[r] & \{e\} 
}
\end{equation}
of short exact sequences, where the vertical maps identify $\Gamma_q$, $\tilde G_q$ and $G_q$ with the subgroups $(\Gamma_p)_{\tilde q}$, $(\tilde G_p)_{\tilde q}$ and $(G_p)_q$ respectively.  Furthermore, there is a commutative diagram
\begin{equation}\label{E:A'}
\xymatrix{
\tilde U_q \ar[r] \ar[d]_{\pi_q} & \tilde U_p \ar[d]^{\pi_p} \\
U_q \ar[r] & U_p
}
\end{equation}
where each map is equivariant with respect to the appropriate actions of the groups $\tilde G_q$, $\tilde G_p$, $G_q$ and $G_p$.
\end{lem}


\begin{proof}  Given any good local chart $U_p$ around $p$, note first that, via inclusion, there is a commutative diagram
\begin{equation}\label{E:B}
\xymatrix{
\{e\}  \ar[r] & (\Gamma_p)_{\tilde q} \ar[r] \ar[d] & (\tilde G_p)_{\tilde q} \ar[r]^{\rho_p|_{(\tilde G_p)_{\tilde q}}} \ar[d] & (G_p)_q \ar[r] \ar[d] & \{e\} \\
\{e\} \ar[r] & \Gamma_p \ar[r] & \tilde G_p \ar[r]^{\rho_p} & G_p \ar[r] & \{e\} 
}
\end{equation}
for every $q \in U_p$ and every $\tilde q \in \pi_p^{-1}(q)$.

Similarly to the case of a Riemannian manifold, $U_p$ can be chosen to be $G_p$-invariant and small enough such that all geodesics emanating from $p$ and orthogonal to the orbit $G(p)$ minimize the distance from $G(p)$.  Given $q \in U_p$, Kleiner's Isotropy Lemma yields $G_q = (G_p)_q$. 

By Lemma \ref{L:localgp}, a good local chart $U_q$ around $q$ with $\tilde{U}_q\In \tilde U_p$ can be chosen such that $\Gamma_q = (\Gamma_p)_{\tilde q}$ and $\pi_q = \pi_p|_{\tilde U_q}$. Furthermore, if $\tilde U_q$ is chosen to be $(\tilde G_p)_{\tilde{q}}$-invariant, then $U_q$ is $(G_p)_q=G_q$-invariant, and the inclusion $U_q\In U_p$ is $G_q$-equivariant. Restricting the actions from $\tilde U_p$, $U_p$ to $\tilde U_q$, $U_q$ respectively, determines a commutative diagram
\begin{equation}\label{E:C}
\xymatrix{
\{e\}  \ar[r] & (\Gamma_p)_{\tilde q} \ar[r] \ar[d]^{=} & (\tilde G_p)_{\tilde q} \ar[r]^{\rho_p|_{(\tilde G_p)_{\tilde q}}} \ar[d] & (G_p)_q \ar[r] \ar[d]^{=} & \{e\} \\
\{e\} \ar[r] & \Gamma_q \ar[r] & \tilde G_q \ar[r]^{\rho_q} & G_q \ar[r] & \{e\} 
}
\end{equation}
where the horizontal arrows are exact. By the Five Lemma, the map $(\tilde G_p)_{\tilde q}\to G_q$ is an isomorphism. Combining the diagrams \eqref{E:B} and \eqref{E:C} yields the diagram in \eqref{E:A}. The equivariance and commutativity of the diagram \eqref{E:A'} follows by construction.
\end{proof}


Recall from Lemma \ref{L:strong} that orbits of Lie group actions are strong suborbifolds.  Using the notation developed in Section \ref{S:Setup}, there is a version of the Slice Theorem for orbifolds (for a proof, see, for example, \cite{Ye}).

\begin{thm}[Slice theorem] 
\label{T:slice-thm}
Suppose that a compact Lie group $G$ acts on an orientable orbifold $\orb$ equipped with a $G$-invariant, orbifold-Riemannian metric, and let $G(p)$ be the orbit of $G$ through $p\in \orb$. 
Then a $G$-invariant neighbourhood of $G(p)$ is equivariantly diffeomorphic to
\[
G \times_{G_p} \left(\tilde{\nu}_p G(p)/\Gamma_p\right)
\]
and, by Proposition \ref{P:LES}, this is equivariantly diffeomorphic to
\[
G\times_{\tilde{G}_p}\tilde{\nu}_p G(p).
\]
\end{thm}

 
\section{Torus orbifolds}
\label{S:TORBIFOLDS}


\begin{defn}
A pair $(\orb^{2n},T^n)$, $n\geq 1$, is a \emph{torus orbifold} if $\orb^{2n}$ is a $2n$-dimensional, closed, oriented orbifold on which the $n$-dimensional torus $T^n$ acts smoothly and effectively with non-empty fixed-point set. 
\end{defn}

To avoid confusion, henceforth the notation $G=T^n$ will be used.  The identity component of a subgroup $K \In G$ will be denoted by $K^o$.  If the action is clear from the context, a torus orbifold $(\orb,G)$ will be denoted simply by $\orb$.  It will always be assumed that $\orb$ is equipped with an invariant orbifold-Riemannian metric (cf. \cite{Bre}).


\begin{defn} 
\label{D:Strata}
Let $\orb$ be a torus orbifold and let $p\in \orb$. The \emph{stratum} containing $p$, which will be denoted by $\Sigma_p$, is the connected component of the set
\[
\{ q\in \orb \ | \ G_q^o=G_p^o\textrm{ as subgroups of }G\}
\]
which contains $p$. The projection $\overline{\Sigma}_p/G\In \orb/G$ of the closure, $\overline{\Sigma}_p$, of a stratum $\Sigma_p$ is called an \emph{(orbifold) face} of $\orb/G$.  A one-dimensional face of $\orb/G$ is called an \emph{edge}.
\end{defn}

It follows from Definition \ref{D:Strata} that the closure $\overline{\Sigma}_p$ of the stratum containing $p$ is a connected component of the fixed-point set $\orb^{G_p^o}$ and hence, by Corollary \ref{C:fixed-pt}, a strong suborbifold of $\orb$.

Note that the identity component $G_p^o$ of an isotropy group $G_p$ is a connected, compact, abelian Lie group, hence a torus. In particular, $\tilde{G}^o_p$ acts effectively on $\tilde{U}_p\cong \mathbb{R}^{2n}$. This fact implies the following lemma. 


\begin{lem}\label{L:fixed points isolated}
The fixed-point set of a torus orbifold $(\orb,G)$ consists of finitely many isolated points.  Hence $H^{\odd}(\orb; \Q) = 0$ if $\orb$ is simply connected and rationally elliptic.
\end{lem} 


\begin{proof}
Let $\orb^G$ be the fixed-point set of the $G$ action, let $p\in \orb^G$, and let $U_p$ be a $G$-invariant good local chart around $p$.  By Corollary \ref{C:fixed-pt}, each component of $\orb^G$ is a strong suborbifold.  From the discussion in Section \ref{S:Setup}, there is a splitting $\tilde T_p U_p = \tilde T_p \orb^G \oplus \tilde \nu_p \orb^G$ and this splitting is $\tilde G$ invariant.  Since the identity component $\tilde G^o$ acts trivially on $\tilde T_p \orb^G$, it must act linearly and effectively on $\tilde \nu_p \orb^G$.  As $\tilde G^o$ is an $n$-dimensional torus, it follows that $\dim \tilde \nu_p \orb^G \geq 2n$.  Since $\dim \tilde T_p U_p = 2n$, one obtains $\dim \tilde T_p \orb^G = 0$.  Therefore, the components of $\orb^G$ are points and, since $\orb$ is compact, $\orb^G$ must be finite.

It now follows from the Euler characteristic identity $\chi(\orb^G) = \chi(\orb)$ (cf. \cite[p. 163]{Bre}, \cite{Kob}) that $\chi (\orb) > 0$.  Whenever $\orb$ is also simply connected and rationally elliptic, this is equivalent to $H^{\odd}(\orb;\Q) = 0$ \cite[p. 444]{FHT}.
\end{proof}


\begin{lem}
\label{L:Standard_rep} Given a torus orbifold $(\orb, G)$, suppose that $p \in \orb^G$ and $U_p$ is a $G$-invariant good local chart around $p$. Then:
\begin{enumerate}
\item 
\label{L1}
The action of $G$ on $U_p$ lifts to an action of $\tilde{G}$ on $\tilde{U}_p$ such that the isotropy action $\tilde{G}^o\times \tilde T_p U_p\rightarrow  \tilde T_p U_p$ is equivalent to the standard $n$-torus action on $\mathbb{C}^n$. 
\item 
\label{L2}
$\Gamma_p$ is a subgroup of $\tilde{G}^o$. In particular, $\tilde{G}=\tilde{G}^o$, i.e. $\tilde{G}$ is connected, hence a torus.
\end{enumerate}
\end{lem}


\begin{proof}
(\ref{L1}) By Lemma \ref{L:fixed points isolated}, the $n$-torus $\tilde{G}^o$ acts effectively on $\tilde T_p U_p\cong\mathbb{C}^{n}$. As there is a unique (up to equivariant diffeomorphism) linear, effective $n$-torus action on $\sph^{2n-1}\In \mathbb{C}^{n}$, i.e. the action of the maximal torus in $\U(n)$, it follows that the $n$-torus $\tilde{G}^o$ acts on $\tilde T_p U_p\cong\mathbb{C}^{n}$ in the standard way. 

(\ref{L2}) Since $\tilde T_p U_p$ can be equipped with a $\tilde{G}$-invariant metric, $\Gamma_p$ and $\tilde{G}^o$ may be considered as subgroups of $\SO(2n)$. By (\ref{L1}), $\tilde{G}^o$ is the maximal torus of $\U(n) \In \SO(2n)$ and, by Corollary \ref{C:commute}, $\Gamma_p$ commutes with $\tilde{G}^o$.  Then $\Gamma_p\In \tilde{G}^o$, since the centralizer of a maximal torus  in $\SO(2n)$ is the maximal torus itself.
\end{proof}


\begin{cor}
\label{C:Corner}
Let $\orb^{2n}$ be a $2n$-dimensional torus orbifold. Fix $p\in \orb^G$ and let $U_p$ be a $G$-invariant good local chart around $p$. Then $U_p/G= \tilde{U}_p/\tilde{G}$ is face-preserving diffeomorphic to $\RR_+^{n} = \{(x_1,\ldots,x_n) \in \RR^n \mid x_i \geq 0, \ i=1,\ldots,n \}$.
\end{cor}


\begin{lem} [cf. {\cite[Lemma~2.2] {MaPa}}]
\label{L:Masuda_Panov_orb} Let $\orb$ be a closed $n$-orbifold with a smooth, effective action by a $k$-torus $G$, $k\leq n$.  Let $H\In G$ be a subtorus and $\mc{N}\In \orb^H$ a connected component of its fixed-point set. If $H^{\mathrm{odd}}(\orb;\Q)=0$, then $H^{\mathrm{odd}}(\mc{N};\Q)=0$ and $\mc{N}^G\neq \emptyset$ (i.e. $\mc{N}\cap \orb^G\neq \emptyset$).
\end{lem}


\begin{proof}
As $H$ acts on $\orb$ with only finitely many orbit types, it follows from \cite[Lemma 4.2.1]{AP} that there is a circle subgroup $K \In H$ with $\orb^K = \orb^H$.  The pair $(\orb, \emptyset)$ satisfies the hypotheses of \cite[Lemma 3.10.13]{AP}, hence 
\[
\sum_{i=0}^\infty \dim_\Q H^{2i+1}(\orb^K;\QQ) \leq  \sum_{i=0}^\infty \dim_\Q H^{2i+1}(\orb;\QQ). 
\]
Now $H^{\mathrm{odd}}(\orb;\Q)=0$ implies that 
\[
H^{\odd}(\orb^H;\QQ) = H^{\odd}(\orb^K;\QQ)=0.
\]
Therefore $H^{\odd}(\mc{N};\QQ)=0$ and, consequently, $\chi(\mc{N})>0$ for all connected components $\mc{N} \In \orb^H$. As $G$ is abelian, each connected component $\mc{N} \In \orb^H$ admits an action of $G$.  The fixed-point set $\mc{N}^G = \mc{N} \cap \orb^G \In \orb^G$ of this action must be non-empty, since $\chi(\mc{N}^G)  = \chi(\mc{N}) > 0$.
\end{proof}



\begin{prop}\label{P:right-codimension}
Let $\orb^{2n}$ be a torus orbifold with $H^{\mathrm{odd}}(\orb;\Q)=0$.  Fix $p\in \orb^{2n}$ and 
let $\overline \Sigma_p$ be the closure of the stratum $\Sigma_p$ in $\orb$.  Then:
\begin{enumerate}
\item $\ol\Sigma_p$ is a codimension-$(2\, \dim G_p)$ torus orbifold with $H^{\odd}(\overline \Sigma_p; \Q) = 0$.
\item The linear, effective action of $\tilde G_p^o$ on $\tilde T_p U_p = \tilde T_p \ol\Sigma_p \oplus \tilde\nu_p \ol\Sigma_p$ is trivial on the first summand and equivalent to the standard $(\dim G_p)$-torus action on $\CC^{\dim G_p}$ on the second.
\end{enumerate}
\end{prop}


\begin{proof}
Equip $\orb^{2n}$ with a $G$-invariant orbifold-Riemannian metric.  Recall from Lemma \ref{L:Masuda_Panov_orb} that $\overline \Sigma_p\In \orb^{G_p^o}$ contains some fixed point $p_0$ of the $G$ action.  Given that $\overline \Sigma_p$ is totally geodesic, being a component of the fixed-point set of an isometric group action, let $\sigma:[0,1]\to \ol\Sigma_p$ be a minimal geodesic in $\orb$ from $\sigma(0) = p$ to $\sigma(1) = p_0$. By definition, $G_p^o \In G_q^o$ for all $q \in \overline \Sigma_p$, hence $G_{\sigma(t)}^o \Ni G_p^o$ for all $t \in [0,1]$.  

On the other hand, since $p_0$ is a fixed point of the action of $G$, $d(p_0, p) = d(p_0, g \cdot p)$ for all $g \in G$, hence $\sigma$ is a minimal geodesic from $p_0$ to $G(p)$.   By Kleiner's Isotropy Lemma, $G_{\sigma(t)} \In G_p$ for all $t \in (0,1)$.  Therefore, $G_{\sigma(t)}^o = G_p^o$ for all $t \in [0,1)$.   

By Lemma~\ref{L:Standard_rep}, the $G$ action on $\orb^{2n}$ is ``standard'' around $p_0$, i.e.~the action lifts via differentials to the unique linear, effective $n$-torus action on $\tilde T_{p_0} U_{p_0}\cong\mathbb{C}^n$, where $U_{p_0}$ is a $G$
-invariant good local chart around $p_0$.

Let $\tilde \sigma$ be a geodesic in $\tilde T_{p_0} U_{p_0}$ such that $\pi_{p_0}(\exp_{p_0} \tilde{\sigma}(t)) = \sigma(t)$ for all $t$ where this makes sense. Up to a permutation of coordinates, $\tilde \sigma(t) = (0,\ldots,0, t\,  z_{l+1} , \ldots,t\, z_n)\in \tilde T_{p_0} U_{p_0}$ for some $l \leq n$. 
By Lemma \ref{L:Isotropy dimension}, $\tilde G_{\sigma(t)}^o \In \tilde G^o = \tilde G$ is an $l$-torus acting on the $l$ trivial coordinates, and the pre-image of $\ol \Sigma_p \cap U_{p_0}$ in $\tilde T_{p_0} U_{p_0}$ is the $2(n-l)$-dimensional fixed-point set of this action.

Notice that the action of $G$ on $\orb$ induces an effective $(n-l)$-torus action on $\overline \Sigma_p$ with fixed points.  Moreover, since $H^{\mathrm{odd}}(\orb;\Q)=0$, Lemma \ref{L:Masuda_Panov_orb} gives $H^{\odd}(\overline \Sigma_p; \Q) = 0$.

In order to verify that $\ol\Sigma_p$ is a torus orbifold, it remains to show that $\ol\Sigma_p$ is oriented.  Recall that $\ol\Sigma_p$ is a strong suborbifold of $\orb$.  Then, given any $q \in \ol\Sigma_p$ and a $G_q$-invariant good local chart $U_q \In \orb$, there is a splitting $\tilde T_q U_q = \tilde T_q \ol\Sigma_p \oplus \tilde \nu_q \ol\Sigma_p$.  As the action of the $l$-torus $\tilde G_q^o$ on $\ol\Sigma_p$ is trivial, it must act effectively on the $2l$-dimensional space $\tilde\nu_q \ol\Sigma_p$.  Hence, this action is equivalent to the standard $T^l$ action on $\CC^l$ and, as such, inherits a natural orientation.  An orientation on $\ol\Sigma_p$ can now be defined which is compatible with those on $\orb$ and $\tilde\nu_q \ol\Sigma_p$.

Finally, setting $q = p$ above yields part (b) of the statement.
\end{proof}


\begin{prop}
\label{P:Bullet_point_3}Let $\orb^{2n}$ be a torus orbifold with $H^{odd}(\orb;\Q)=0$. Then every point $p\in \orb^{2n}$ lies in the closures of exactly $\dim(G_p)$ strata of codimension $2$. Equivalently, a point $[p]\in \orb/G$ in the (relative) interior of a face of codimension $k$, lies in  exactly $k$ faces of codimension $1$ in $\orb/G$.
\end{prop}


\begin{proof}

Fix $p$ in $\orb$ and let $U_p$ denote some small $G_p$-invariant good local chart around $p$, such that at every $q\in U_p$, $G_q=(G_p)_q$ (cf. Lemma \ref{L:Isotropy dimension}). Then a fixed-point set $\orb^H$, for some $H\In G$, intersects $U_p$ if and only if $H\In G_p$. In particular, $p$ belongs to the closure of each stratum intersecting $U_p$. 

By Lemma \ref{L:Isotropy dimension}, the closure $\ol{\Sigma}=\orb^H$ of a stratum intersecting $U_p$ lifts, via $\pi_p:\tilde{U}_p\to U_p$, to the fixed-point set $\tilde{U}_p^{\tilde{H}}$ where $\tilde{H}$ is the identity component of $\rho^{-1}(H)$  and $\rho:\tilde{G}_p\to G_p$ is the map constructed in Proposition \ref{P:LES}. In particular, there is a one-to-one correspondence between codimension-2 fixed-point sets in $\tilde{U}_p$ of subgroups of $\tilde{G}_p^o$ and codimension-2 strata in $\orb$ intersecting $U_p$  whose closure contains $p$. By Proposition \ref{P:right-codimension}(b), there are precisely $l=\dim G_p$ different codimension-2 fixed-point sets in $\tilde{U}_p$, each fixed by a different $S^1$-factor of $\tilde{G}_p^o=T^l$. This then finishes the proof. \end{proof}


\begin{lem}
\label{L:No_fake} Let $\orb$ be a torus orbifold with $H^\odd (\orb; \Q) = 0$. Then the closure of each two-dimensional stratum of $\orb$ is homeomorphic to a two-sphere and each one-dimensional face (edge) in the quotient $\orb/G$ contains exactly two fixed points.
\end{lem}


\begin{proof}
Recall that the closure $\ol \Sigma^2_i$ of each  two-dimensional stratum $\Sigma^2_i$ in a torus orbifold $\orb$ projects down to a one-dimensional face of $\orb/G$. By Proposition~\ref{P:right-codimension}, each $\ol \Sigma^2_i$ contains a fixed point of the $G$ action and is a two-dimensional torus orbifold with $H^\odd(\ol \Sigma^2_i;\Q) = 0$. By \cite[Chap. 13]{Thu} the $\overline{\Sigma}^2_i$ are closed, orientable, topological $2$-manifolds with positive Euler characteristic, hence each must be homeomorphic to a two-dimensional sphere. Therefore each $\overline{\Sigma}_i^2$ has Euler characteristic $2$ and contains exactly two fixed points of the $G$ action. 
\end{proof}


\section{Weights, GKM-graphs and the moment-angle complex} 
\label{SS:MANGLECPX}

Let $\orb$ be a $2n$-dimensional torus orbifold with
\(H^\text{odd}(\orb;\Q)=0\). A \emph{facet} of the orbit space
$Q=\orb/G$ is a face of codimension one. Recall that, by Proposition
\ref{P:Bullet_point_3}, this corresponds, in \(\orb\), to the closure $\ol\Sigma_p$ of a stratum $\Sigma_p$ defined by a one-dimensional isotropy group $G_p$.

Given a facet $F$, let $p \in \orb$ be a point with $\dim(G_p) = 1$ such that $\ol\Sigma_p$ is the pre-image of $F$, and let $U_p\subset \orb$ be a $G_p$-invariant good local chart around $p$.

Formally assign a circle $S^1_F$ to $F$ and let the \emph{label}
\[
\lambda_F:S^1_F \to G
\]
denote the composition (covering)
\[
S^1_F \stackrel{\cong}{\longrightarrow} \tilde G_p^o \stackrel{\rho_{p}}{\longrightarrow} G_p^o \In G,
\]
where the map $\rho_p : \tilde G_p \to G_p$ is that arising in Proposition \ref{P:LES}. Set now $T_Q=\prod_FS^1_F$ and define the \emph{label map}
\[
\lambda = \prod_F \lambda_F:T_{Q}\longrightarrow G.
\] 


\begin{lem}
The label map $\lambda: T_Q \to G$ is well defined.
\end{lem}


\begin{proof}
In order to verify that the map $\lambda$ is well defined, it need
only be demonstrated that the labels $\lambda_F$ do not depend on the
choice of the point in the pre-image $\ol\Sigma_p$ of a facet $F$.  By
Definition \ref{D:Strata}, if $q \in \ol\Sigma_p$ is another point with $\dim(G_q) = 1$, then $G_p^o = G_q^o$ as subgroups of $G$.  It suffices to show that there is an isomorphism $\alpha : \tilde G_q^o \to \tilde G_p^o$ such that the following diagram commutes:
$$
\xymatrix{
S^1_F \ar[r]^{\cong} \ar[dr]_{\cong} & \tilde G_p^o \ar[r]^{\rho_p} & G_p^o  \\
& \tilde G_q^o \ar[r]^{\rho_q} \ar@{.>}[u]^{\alpha} & G_q^o \ar[u]^{=} 
}
$$
As each of $G_p^o$, $G_q^o$, $\tilde G_p^o$ and $\tilde G_q^o$ is a circle, if the kernels of $\rho_p$ and $\rho_q$ have the same order, then it is possible to lift the identity $G_q^o \stackrel{=}{\lra} G_p^o$ to such an isomorphism $\alpha$.

The kernels of $\rho_p$ and $\rho_q$ are given by $\Gamma_p \cap \tilde G_p^o$ and $\Gamma_q \cap \tilde G_q^o$ respectively.  Since the stratum $\Sigma_p$ is connected, it is enough to show that the order of \(\Gamma_p\cap \tilde{G}_p^o\) is locally constant.

Let \(U_p\) be a sufficiently-small $G_p$-invariant good local chart around \(p\) such that \(\tilde{U}_p\) is a linear \(\tilde{G}_p\)-representation.  Then \(\tilde{U}_p\) is of the form \(V\oplus W\) such that \(\tilde{G}_p^o\) acts non-trivially on \(V\) and trivially on \(W\) (see Proposition~\ref{P:right-codimension}).

Moreover, since \(\Sigma_p\) has codimension two, it follows that \(V\) is two dimensional and \(\pi_p(W)=\Sigma_p\cap U_p\).  Therefore, the subgroup \(\Gamma_p\cap \tilde{G}_p^o\) of \(\Gamma_p\) acts trivially on \(W\).  For any $q \in \Sigma_p \cap U_p$ and $\tilde{q} \in \pi_p^{-1}(q) \cap W$ one has $\Gamma_p \cap \tilde G_p^o \In (\Gamma_p)_{\tilde q}$. By Lemmas \ref{L:localgp} and \ref{L:Isotropy dimension}, $\Gamma_q=(\Gamma_p)_{\tilde{q}}$ and $\tilde G_q^o=\tilde G_p^o$, hence $\Gamma_p \cap \tilde G_p^o \In \Gamma_q\cap \tilde G_q^o$. On the other hand, the same lemmas yield $\Gamma_q\In \Gamma_p$ and $\tilde G_q^o\In \tilde G_p^o$, hence $\Gamma_q \cap \tilde G_q^o \In \Gamma_p\cap \tilde G_p^o$. Therefore, $\Gamma_p \cap \tilde G_p^o$ is locally constant. 
\end{proof}

The labels of the facets can be used to define weights on the edges of
the orbit space. By Proposition \ref{P:Bullet_point_3}, any edge $E$
is the intersection of $n-1$ facets $\{F_1,\ldots F_{n-1}\}$ and, by
restricting the label map to $T_E=\prod_{i=1}^{n-1}S^1_{F_i}$, one
obtains a homomorphism $\lambda_E:T_E\to G$. 
Let $p_i$ be a generic point in the stratum corresponding to \(F_i\).
 As $S^1_{F_i} \to
G_{p_i}^o \In G$ is a covering, for all facets $F_i$, the map $\lambda_E$ induces an injective map $\mf{t}_E \to \g$ on Lie algebras, hence a surjective map $\g^* \to \mf t_E$ on the corresponding dual spaces.  Since the dual $\mf l^*$ of the Lie algebra of a Lie group $L$ is canonically isomorphic to $H^2(BL;\RR)$, one concludes that the induced map $\lambda_E^* : H^2(BG;\ZZ)=\ZZ^n \to H^2(BT_E; \ZZ)=\ZZ^{n-1}$ has full rank.  Define the \emph{weight} $\mu(E)\in H^2(BG;\ZZ)$ of $E$ to be a generator of the kernel of $\lambda_E^*$.

 In this way, one obtains a system of weights on the vertex-edge graph of the orbit space $Q$, i.e.~on the union of edges and vertices. This is the well-known GKM-graph associated to the torus orbifold $\orb$. In an analogous manner to the manifold case (cf.~\cite{MMP}), this graph determines the rational equivariant cohomology ring \(H_G^*(\orb;\Q) = H^*(\orb_G;\Q)\) of $\orb$, where $\orb_G = \orb \times_G EG$ is the Borel construction, in the following way: Since \(H^\text{odd}(\orb;\Q)=0\) by assumption, it follows that \(H_G^*(\orb;\Q)\) is a free \(H^*(BG;\Q)\)-module, which is easily seen from the spectral sequence of the homotopy fibration $\orb \to \orb_G \to BG$.  In particular, the induced homomorphism $H_G^*(\orb;\mathbb{Q}) \to H^*(\orb;\Q)$ is surjective.  

If $\orb^{(1)} \In \orb$ is the union of all $G$-orbits of dimension at most one, i.e. the pre-image of the vertex-edge graph of $Q$, the respective inclusion maps induce a commutative diagram
$$
\xymatrix{
H_G^*(\orb^{(1)}; \Q)  \ar[r] & H_G^*(\orb^G; \Q) \\
H_G^*(\orb; \Q) \ar[u] \ar[ur] &
}
$$
It follows from Lemma 2.3 and Proposition 2.4 of \cite{ChSk} that the homomorphisms $H_G^*(\orb;\Q)\rightarrow H_G^*(\orb^G;\Q)$ and $H_G^*(\orb^{(1)};\Q)\to H_G^*(\orb^G;\Q)$ have the same image, and the former homomorphism is injective.  Furthermore, the homomorphism $H_G^*(\orb; \Q) \to H_G^*(\orb^{(1)}; \Q)$ must, therefore, also be injective.

By Lemma \ref{L:No_fake}, \(\orb^{(1)}\) is a union of two-dimensional spheres (intersecting only in the fixed points of the $G$ action).  Therefore, the following theorem follows as in the manifold case \cite[Theorem 7.2]{GKM}:


\begin{thm}
\label{T:EqCoh}
Let \(\orb\) be a torus orbifold with fixed points $\{p_1, \dots, p_N\}$ and $H^\odd(\orb;\Q) = 0$. Then, via the natural restriction map
\[
H^*_G(\orb;\Q)\to H^*_G(\orb^G;\Q) = \bigoplus_{i=1}^N H^*(BG;\Q),
\] 
the equivariant cohomology algebra $H^*_G(\orb;\Q)$ is isomorphic to the set of $N$-tuples $(f_1,\dots, f_N) \in H^*_G(\orb^G;\Q)$, with the property that if the vertices $p_i$ and $p_j$ in the associated GKM-graph are joined by an edge with weight $\mu_{ij} \in H^2(BG;\Q)$, then $f_i - f_j$ lies in the ideal of $H^*(BG;\Q)$ generated by $\mu_{ij}$.
\end{thm}


\begin{rem}\label{R:functorial}
The process by which one obtains $H_G^*(\orb; \Q)$ from the GKM-graph is functorial in the following sense:  Suppose that $\orb$ (resp. $\hat \orb$) is a $2n$-dimensional torus orbifold with fixed points $p_1, \dots, p_N$ (resp. $\hat p_1, \dots, \hat p_{\hat N}$) and weights $\mu$ (resp. $\hat \mu$).  Suppose, further, that there is a weight-preserving, injective map $\vphi$ between the GKM-graphs of $\orb$ and $\hat\orb$, i.e. $\hat\mu(\vphi(E)) = \mu(E)$, for each edge $E$ of $\orb/G$.

The map $\vphi$ induces an injective homomorphism
\[
\vphi_\# : \bigoplus_{i=1}^N H^*(BG;\Q) \to \bigoplus_{i=1}^{\hat N} H^*(BG;\Q),
\]
where, for each $i_0 \in \{1, \dots, N\}$ and given $\vphi(p_{i_0}) = \hat p_{j_0}$, the restriction of $\vphi_\#$ to the $i_0$-th summand of $\bigoplus_{i=1}^N H^*(BG;\Q)$ is given by the identity map onto the $j_0$-th summand of the target space $\bigoplus_{i=1}^{\hat N} H^*(BG;\Q)$.

By Theorem \ref{T:EqCoh}, $H_G^*(\orb; \Q)$ and $H_G^*(\hat \orb; \Q)$ embed into $\bigoplus_{i=1}^N H^*(BG;\Q)$ and $\bigoplus_{i=1}^{\hat N} H^*(BG;\Q)$, respectively.  Since $\vphi$ is weight-preserving, $\vphi_\#$ maps $H_G^*(\orb; \Q)$ into $H_G^*(\hat \orb; \Q)$.  It then follows that there is an induced $H^*(BG; \Q)$-module homomorphism 
\[
H_G^*(\orb; \Q) \to H_G^*(\hat \orb; \Q),
\]
which, by abuse of language, will be denoted also by \(\varphi_\#\).
Moreover, if $H^{\odd}(\orb; \Q) = H^{\odd}(\hat \orb; \Q) = 0$, then 
\[
H^*(\orb; \Q) = H_G^*(\orb;\Q) / H^{> 0}(BG; \Q) \cdot H_G^*(\orb; \Q)
\] 
and similarly for $H^*(\hat \orb; \Q)$.  Hence, there is an induced homomorphism
\[
\ol\vphi_\# : H^*(\orb; \Q) \to H^*(\hat \orb; \Q)
\]
such that the diagram
\begin{equation}
\label{E:EqRHE}
\xymatrix{
H_G^*(\orb; \Q) \ar[d] \ar^{\vphi_\#}[r] & H_G^*(\hat \orb; \Q) \ar[d] \\
H^*(\orb; \Q) \ar_{\ol\vphi_\#}[r] & H^*(\hat \orb; \Q)
}
\end{equation}
commutes.
\end{rem}
\bigskip


Recall now that an $n$-dimensional manifold with corners $Q$, i.e. a
manifold locally modelled on $\RR^n_+$, is called \emph{nice} if each
one of its codimension-$k$ faces is contained in  exactly $k$ facets,
i.e. codimension-$1$ faces, of $Q$.

Formally assign a copy $S^1_F$ of the circle to each facet $F$ of $Q$ and let $T_Q= \prod_F S_F^1$ be the torus given by their product.

For any $q\in Q$, let $T(q) = \prod_{F \ni q} S^1_F \In T_Q$ denote the subtorus generated by the circles corresponding to the facet of $Q$ which contain $q$. The \emph{moment-angle complex} is defined by $Z_{Q}= (Q\times T_Q)/\!\!\sim$, where $(q_1,t_1)\sim (q_2,t_2)$ if $q_1=q_2$ and $t_1t_2^{-1}\in T(q_1)$.

As $Q$ is a nice manifold with corners, it follows that $Z_Q$ is a topological manifold with a continuous $T_Q$ action, such that $Z_Q/T_Q$ is homeomorphic to $Q$. 

Suppose that, in addition, $Q$ has $0$-dimensional faces.  Consider a torus $G=T^n$ and a homomorphism
\[
\hat\lambda: T_Q \to G
\]
such that, for every $q \in Q$, the restriction $\hat\lambda|_{T(q)} : T(q) \to G$ has finite kernel.  This condition ensures that the kernel $K$ of $\hat\lambda$ acts almost freely on $Z_Q$.  The group $G$ then acts on the quotient $\orb_Q = Z_Q/K$ such that $(\orb_Q, G)$ is a $2n$-dimensional torus orbifold whose orbit space $\orb_Q/G$ has labels induced by the assignment $\hat\lambda$, and there is a face-preserving homeomorphism $\orb_Q/G\to Q$.

The following three standard examples will be needed in the proof of Theorem \ref{T:main_thm}.


\begin{example}
\label{eg:simplex}
 \cite[Ex.\ 6.7]{BuPa}
If $Q = \Delta^n$ is an $n$-dimensional simplex, then $T_Q$ is an $(n + 1)$- dimensional torus. Moreover, $Z_Q$ is equivariantly homeomorphic to $\sph^{2n+1} \In \CC^{n+1}$ with the standard linear torus action.
\end{example}


\begin{example}
\label{eg:susp} \cite[Ex.\ 4.3]{MaPa}
If $Q = \Sigma^n$ is the suspension of the simplex $\Delta^{n-1}$, then $T_Q$ is $n$-dimensional. Moreover, $Z_Q$ is equivariantly homeomorphic to $\sph^{2n} \In \CC^n\times\RR$ with the standard linear torus action.
\end{example}


\begin{example}
\label{eg:product}
\cite[Prop.\ 6.4]{BuPa}
Let $Q_1$ and $Q_2$ be two nice manifolds with corners. If $Q = Q_1 \times Q_2$, then $T_Q = T_{Q_1} \times T_{Q_2}$ and $Z_Q$ is equivariantly homeomorphic to $Z_{Q_1} \times Z_{Q_2}$ .
\end{example}


\section{Equivariant classification of torus orbifolds}
\label{S:PROOF}

In order to prove Theorem  \ref{T:main_thm}, it is necessary to first understand the combinatorial properties of the  face poset of the orbit space $\orb/G$.


\begin{prop}\label{P:comb-properties}
Let $\orb$ be a simply-connected, rationally-elliptic torus orbifold. Then the face poset of $\orb/G$ satisfies:
\begin{enumerate}
	\item The vertex-edge graph of each face is connected. 
	\item Each face of $\orb/G$ contains at least one vertex.
	\item Each face of $\orb/G$ of codimension $k$ is contained in exactly $k$ faces of codimension $1$. 
	\item Each one-dimensional face of $\orb/G$ contains exactly two fixed points of the $G$ action. 
	\item Every two-dimensional face of $\orb/G$ contains at most four vertices. 
	\item For $d\geq 3$, no $d$-dimensional face is combinatorially equivalent to the face poset of $[-1,1]^d/\{\pm \id\}$. 
\end{enumerate}
\end{prop}


\begin{proof}
Property (a)  follows from \cite[Prop.~2.5]{ChSk}. Indeed, $H^\odd(\orb;\Q)=0$ and $H^\odd(BG;\Q)=0$ imply that the differentials in the spectral sequence of the homotopy fibration  $\orb \to \orb_G \to BG$ are trivial. Therefore $H^*_G(\orb;\Q)=H^*(\orb_G;\Q)=H^*(\orb;\Q)\otimes H^*(BG;\Q)$, thus fulfilling the hypotheses of the aforementioned proposition.  Properties (b), (c) and (d) have been verified in Lemmas~\ref{L:Masuda_Panov_orb}, \ref{P:Bullet_point_3} and \ref{L:No_fake}, respectively. 


To see that property (e) holds, one must modify the proof of Lemma 4.2 of \cite{Wi} for the case of torus orbifolds only slightly. The original proof invokes \cite[Corollary 3.3.11]{AP} which, although stated only for rationally-elliptic $G$-CW-complexes,  also holds for compact spaces with finitely many orbit types, e.g.~torus orbifolds, as indicated on page 160 of \cite{AP}.

Finally, suppose that there is a $d$-dimensional face $F$ combinatorially equivalent to the face poset of $X:=[-1,1]^d/\{\pm \id\}$. Notice first that the standard linear, effective $T^d$ action on $(\sph^2)^d$ commutes with the diagonal antipodal map and, therefore, induces an effective $T^d$ action on $N=(\sph^2)^d/\ZZ_2$ with orbit space $X$. Thus, the quotient of the $T^d$ action on the pre-image of $F$ is combinatorially equivalent to the quotient of the $T^d$ action on $N$ and, in particular, the corresponding GKM-graphs are isomorphic. By 
the discussion before Theorem \ref{T:EqCoh},
their rational cohomology rings are the same. However, the pre-image of $F$ is rationally elliptic by \cite[Cor.\ 3.3.11]{AP}, while, on the other hand, $N$ is not:  Indeed, by \cite[Thm.\ 2.4]{Bre}, $H^*\left(N;\QQ\right)=H^*\big((\sph^2)^d;\QQ\big)^{\ZZ_2}$ and therefore the Betti numbers of $N$ satisfy $b_1(N)=b_2(N)=b_3(N)=0$, $b_4(N)=d(d-1)/2$. In particular $\dim_{\QQ} (\pi_4(N)\otimes\QQ)=b_4(N)$ and, if $N$ were rationally elliptic, Theorem 32.6 in \cite{FHT} would yield
\[
2d(d-1)=4\dim_\QQ (\pi_4(N)\otimes\QQ)\leq \sum_j 2j\dim_\QQ(\pi_{2j}(N)\otimes\QQ)\leq 2d
\]
which is not possible for $d>2$.
\end{proof}


\begin{proof}[Proof of Theorem \ref{T:main_thm}]
Following the arguments involved in proving \cite[Prop. 4.5]{Wi}, the properties established in Proposition \ref{P:comb-properties} are precisely those required to prove that the face poset of $\orb/G$ is combinatorially equivalent to the face poset of $Q=\prod_i \Delta^{n_i}\times \prod_j \Sigma^{n_j}$ as  in Examples \ref{eg:simplex}-\ref{eg:product}, i.e.~there is an isomorphism of face posets $\varphi: \mathcal{P}(\orb/G )\to \mathcal{P}(Q)$.  For each facet $F \in \mc{P}(Q)$, fix an isomorphism $\iota_F : S^1_F \to S^1_{\vphi^{-1}(F)}$.

With $Q$ as above, the moment-angle complex $Z_Q$ of $Q$, together with the action of $T_Q$ as discussed in Section \ref{SS:MANGLECPX}, is equivariantly homeomorphic to a product of spheres $\prod \sph^{n_i}$ equipped with a linear action.

The isomorphism $\varphi$ induces a label map $\vphi_* \lambda : T_Q \to G$, such that the restriction to each factor $S^1_F$ is given by $\lambda_{\varphi^{-1}(F)} \circ \iota_F$.  By setting $\hat\lambda = \vphi_* \lambda$, one can construct, as before, a torus orbifold $(\orb_Q, G)$, where $\orb_Q$ is the quotient of $\prod \sph^{n_i}$ by the linear and almost-free action of a subtorus of $T_Q$ complementary to $G$.

This is achieved as follows: The kernel $\hat{L}$ of   $\hat{\lambda}$ acts almost freely on $\prod \sph^{n_i}$, although it may not be connected. 
Therefore, $\prod \sph^{n_i}/\hat{L}$ is a torus  orbifold. 
Moreover, the identity component $\hat{L}^o$ of $\hat{L}$ is a  subtorus of $T_Q$.

Since the natural action of the finite group $\hat{L}/\hat{L}^o$ on $\prod\sph^{n_i}/\hat{L}^o$ extends to an action of the connected group \(T_Q/\hat{L}^o\), the induced action on cohomology is trivial.
 Hence, by \cite[Thm. 2.4]{Bre}, $\prod \sph^{n_i}/\hat{L}$ and $\prod\sph^{n_i}/\hat{L}^o$ have isomorphic rational cohomology rings.
Moreover, by Proposition~32.16 of \cite{FHT} and Corollary~2.7.9 of \cite{AP}, the minimal models of these spaces are formal and, therefore, isomorphic.
Hence, it may be assumed that $\hat{L}$ is connected.  In this case, define $\orb_Q=\prod \sph^{n_i}/\hat{L}$.

By construction, the torus orbifolds $(\orb,G)$ and $(\orb_Q,G)$ have
isomorphic labelled face posets, hence isomorphic GKM-graphs.
Therefore the rational cohomology rings of $\orb$ and $\orb_Q$  are
isomorphic, as discussed after Theorem~\ref{T:EqCoh}. But once again, the minimal models of these spaces are formal by Proposition~32.16 of \cite{FHT} and Corollary~2.7.9 of \cite{AP}. Since their cohomology rings are isomorphic, this implies that the spaces are rationally homotopy equivalent.

Since $H^*(\orb_G;\Q)=H^*(\orb;\Q)\otimes H^*(BG;\Q)$ and the minimal models of $\orb$ and $\orb_Q$ are formal, it follows again from Corollary~2.7.9 of \cite{AP} that the minimal models of $\orb_G$ and $(\orb_Q)_G$ are formal.  As $H^*(\orb_G;\Q)$ is isomorphic to $H^*((\orb_Q)_G; \Q)$, this ensures that the minimal models of the Borel constructions $\orb_G$ and $(\orb_Q)_G$ are isomorphic, hence $\orb_G \simeq_\Q (\orb_Q)_G$.

Furthermore, from the face-poset isomorphism $\vphi$, one obtains a commutative diagram as in \eqref{E:EqRHE}, where the horizontal arrows are clearly the isomorphisms induced by the rational homotopy equivalence $\orb \simeq_\Q \orb_Q$.  Therefore, $\orb$ is $G$-equivariantly rationally homotopy equivalent to $\orb_Q$.

Finally, if $\orb$ is a (torus) manifold, then all local groups are trivial and one can identify $\tilde{U}_p$ with $U_p$, $\tilde{G}_p$ with $G_p$, and so on.  Given any $p \in \orb$ with $\dim(G_p) = l$, Proposition \ref{P:Bullet_point_3} states that $p$ belongs to the closures $\ol{\Sigma}_i$, $i=1,\ldots, l$, of $l$ codimension-2 strata. By Proposition \ref{P:Bullet_point_3} again, each $\ol{\Sigma}_i$ is fixed by a different factor $S^1_i$ of $G_p^o=T^l$, and the $\ol{\Sigma}_i$ project to distinct facets $F_i$ of $\orb/G$.

By definition, $\lambda_{F_i}:S^1_{F_i}\to G$ sends $S^1_{F_i}$ isomorphically into $S^1_i\In G_p^o$ and, therefore, the label map $\lambda$ sends $T([p])=\prod_{i}S^1_{F_i}$ isomorphically into $G_p^o=\prod_iS^1_i$, where $[p]\in \orb/G$ is the image of $p$.  

In particular, the restriction of $\lambda$ to $T([p])$ has trivial kernel.  It then follows that the kernel of $\lambda$ has trivial intersection with each such torus $T([p])$.  Since the label map $\vphi_*\lambda$ has the same properties as $\lambda$, the kernel of $\vphi_*\lambda$ has trivial intersection with all isotropy subgroups of the $T_Q$ action on $Z_Q$ and, therefore, acts freely on $Z_Q$. Hence, $\orb_Q$ is also a manifold.

\end{proof}


\section{A family of examples}
\label{S:Example}


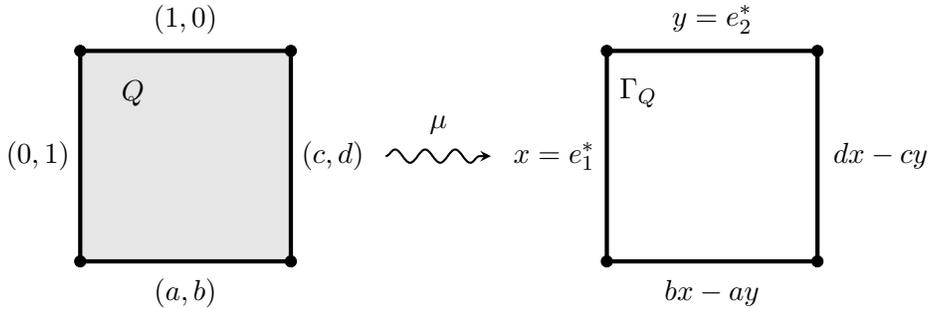
\begin{figure}[h]
  \begin{tikzpicture}[scale=1.4]
    \path[coordinate] (-1.5,1)  coordinate(A)
                (-3.5,1) coordinate(B)
                (-3.5,-1) coordinate(C)
                (-1.5,-1) coordinate(D);
    \fill[color = gray!20!white] (A) -- (B) -- (C) -- (D) -- cycle;

    \draw[ultra thick] (A) -- (B) -- (C) -- (D) -- (A);
    \draw[color = black] (-2.5,1.3) node{$(1,0)$};
    \draw[color = black] (-3.9, 0) node{$(0,1)$};
    \draw[color = black] (-2.5,-1.3) node{$(a, b)$};
    \draw[color = black] (-1.1, 0) node{$(c,d)$};
    \draw[color = black] (-3, 0.6) node{$Q$};
    \filldraw[black] (A) circle (1.5pt)
                     (B) circle (1.5pt)
                     (C) circle (1.5pt)
                     (D) circle (1.5pt);
   
    \path[coordinate] (3.5,1)  coordinate(E)
                (1.5,1) coordinate(F)
                (1.5,-1) coordinate(G)
                (3.5,-1) coordinate(H);

    \draw[ultra thick] (E) -- (F) -- (G) -- (H) -- (E);
    \draw[color = black] (2.5,1.3) node{$y = e_2^*$};
    \draw[color = black] (1, 0) node{$x = e_1^*$};
    \draw[color = black] (2.5,-1.3) node{$bx-ay$};
    \draw[color = black] (4.1, 0) node{$dx-cy$};
    \draw[color = black] (1.8, 0.6) node{$\Gamma_Q$};
    \filldraw[black] (E) circle (1.5pt)
                     (F) circle (1.5pt)
                     (G) circle (1.5pt)
                     (H) circle (1.5pt);
  \draw[thick, decorate, decoration = {snake, segment length = 4mm}, -stealth] (-0.6,0) -- (0.4,0);
  \draw[color = black] (-0.1, 0.3) node{$\mu$};
  \end{tikzpicture}
  \caption{(l) Labelled orbit space $Q$; (r) GKM-graph $\Gamma_Q$}
  \label{F:squares}
  \end{figure}


The family of examples in this section shows the necessity of including almost-free actions in the conclusion of Theorem \ref{T:main_thm}, and also gives an explicit demonstration of how to apply the GKM algorithm discussed in Section \ref{SS:MANGLECPX}.

Consider a $2$-torus $G$ and a $2$-dimensional nice manifold with corners, $Q$, whose boundary consists of four segments labelled as in the square on the left of Figure~\ref{F:squares}, where $a, b, c, d \in \ZZ$.  In this example the four edges and facets of $Q$ coincide.  The labels of the facets of $Q$ are 
the slopes $\in \ZZ^2$ corresponding to circle subgroups (tori of codimension one) in $G$.  By the discussion following Example \ref{eg:product}, in order to construct a $4$-dimensional torus orbifold with orbit space $Q$ the corresponding labels must be linearly independent whenever two facets intersect.  Assume therefore that 
\[
a, d, \det \!\begin{pmatrix} a & b \\ c & d \end{pmatrix} \neq 0.
\]
Since $Q = [0,1]^2$, it follows from Examples~\ref{eg:simplex} and \ref{eg:product} that the moment angle complex $Z_Q$ is equivariantly homeomorphic to $\sph^3 \times \sph^3$ equipped with the standard linear $T^4$ action.  There is, moreover, a surjective homomorphism $T^4 \to G$ whose kernel is a $2$-torus $K$ acting almost freely on $Z_Q \cong \sph^3 \times \sph^3$.  The resulting orbifold $\orb_Q = Z_Q/K \cong (\sph^3 \times \sph^3)/K$ is a simply-connected, rationally-elliptic, $4$-dimensional torus orbifold whose labelled orbit space $Q$ under the action of $G$ is as on the left of Figure~\ref{F:squares}.  Indeed, in this case the action of $K$ on $\sph^3 \times \sph^3$ can be written explicitly as
\begin{align*}
K \times (\SU(2) \times \SU(2)) &\to  \SU(2) \times \SU(2) \\
((z,w), (A,B)) &\mapsto 
\begin{pmatrix} 
\diag(z^{1-a} \bar w^c, 1) A \diag(z^a w^c, \bar z) \\
\diag( \bar z^b w^{1-d}, 1) B \diag( z^b w^d, \bar w)
\end{pmatrix}.
\end{align*}

It remains to demonstrate that not all such labelled orbit spaces $Q$ can be realised by torus manifolds, hence that one cannot always find a torus manifold which is rationally homotopy equivalent to a given torus orbifold.  This will be achieved by computing the cohomology ring and intersection form of the torus orbifold $\orb_Q$.

As discussed in Section~\ref{SS:MANGLECPX}, each edge $E$ of the labelled orbit space $Q$ can be assigned a weight $\mu(E) \in H^2(BG;\ZZ)$.  The resulting GKM-graph $\Gamma_Q$ is shown on the right of Figure~\ref{F:squares}.  As there are four vertices (corresponding to the fixed points of the $G$ action on $\orb_Q$), Theorem~\ref{T:EqCoh} implies that $H^*_G (\orb_Q; \Q)$, the equivariant cohomology algebra of $\orb_Q$, is isomorphic to the set of all $4$-tuples $(f_1, f_2, f_3, f_4) \in \bigoplus_{i=1}^4 H^*(BG;\Q) = \bigoplus_{i=1}^4 \Q[x,y]$ satisfying the relations
\begin{align*}
f_1 - f_2 &= m_1 y, \\
f_2 - f_3 &= m_2 x, \\
f_3 - f_4 &= m_3(bx-ay), \textrm{ and}\\
f_4 - f_1 &= m_4(dx -cy),
\end{align*}
where $m_1, m_2, m_3, m_4 \in \Q[x,y]$, and the ring structure is given by coordinate-wise multiplication.  It is straightforward to check that the equivariant cohomology of $\orb_Q$ is then generated as a $\Q[x,y]$-module by ${\bf 1} = (1,1,1,1)$, $u = (0,-ay, bx-ay, 0)$, $v = (0, -cy, dx-cy, dx-cy)$ and $w = (0, xy, 0, 0)$, of degree $0,2,2$ and $4$ (in $H^*_G (\orb_Q; \Q)$) respectively.  Clearly ${\bf 1}$ is the unit element.

As $H^{\odd}_G (\orb_Q; \Q) = 0$ (Lemma~\ref{L:fixed points isolated}) and, hence, $H^*_G (\orb_Q; \Q)$ is a free $H^*(BG;\Q)$-module, it follows that the rational cohomology of $\orb_Q$ is given by
\[
H^*(\orb_Q;\Q) = H^*_G (\orb_Q; \Q) / (R^+ \cdot H^*_G (\orb_Q; \Q)),
\]
where $R^+ = H^{>0}(BG;\Q)$ and $R^+ \cdot H^*_G (\orb_Q; \Q)$ is the set of all $4$-tuples of the form $m_1 {\bf 1} + m_2 u + m_3 v + m_4 w$, for polynomials $m_1, \dots, m_4 \in \Q[x,y]$ with zero constant term.  Therefore, letting $\alpha$, $\beta$ and $\gamma$ in $H^*(\orb_Q;\Q)$ be the classes represented by $u$, $v$ and $w$ respectively, $H^2(\orb_Q;\Q)$ is generated (over $\Q$) by $\alpha$ and $\beta$, $H^4(\orb_Q;\Q)$ by $\gamma$, and $H^{i}(\orb_Q;\Q) = 0$, $i \neq 0,2,4$.  Moreover, the ring structure is given by the relations
\[
\alpha^2 = ab \gamma, \ \beta^2 = cd \gamma, \ \textrm{and} \ \alpha \beta = ad \gamma.
\]
Indeed, this implies that $\alpha (d\alpha - b\beta) = 0$ and $\beta(a \beta - c \alpha) = 0$.  

Whenever either $b = 0$ or $c = 0$, it is easy to see that one can find generators $\tilde \alpha, \tilde \beta \in H^2(\orb_Q;\Q)$ such that $\tilde \alpha^2 = 0$, $\tilde \beta^2 = 0$ and $\gamma = \tilde \alpha \tilde \beta$.

On the other hand, if $bc \neq 0$ (by assumption $ad(ad-bc) \neq 0$), then the generators $\tilde \alpha = \tfrac{1}{b}\alpha$ and $\tilde \beta = \tfrac{a}{b}(d \alpha - b \beta)$ satisfy 
\[
\tilde \alpha \tilde \beta = 0 \ \textrm{and} \ 
\tilde \beta^2 + ad(ad-bc) \tilde \alpha^2 = 0.
\]
Furthermore, $\tilde \alpha^2$ generates $H^4(\orb_Q;\Q)$ and the intersection form is given by $\diag(1, -ad(ad-bc))$.

However, $\sph^4$, $\CP^2$, $\CP^2 \# \pm \CP^2$ and $\sph^2 \times \sph^2$ are the only closed, simply-connected, smooth, rationally-elliptic $4$-manifolds.  Therefore, if $\orb_Q$ were to be rationally homotopy equivalent to such a manifold, it would have intersection form either $\diag(1, \pm 1)$ or $\left(\begin{smallmatrix} 0&1 \\ 1&0 \end{smallmatrix}\right)$, corresponding to $\CP^2 \# \pm \CP^2$ or $\sph^2 \times \sph^2$.  This is clearly not true for generic $a,b,c,d \in \ZZ$.


\section{Slice-maximal torus actions}
\label{S:CoreThm}


The goal of this section is to prove Theorem \ref{T:CORE_THM}.  To that end, let $M$ be a closed, smooth, simply-connected, rationally-elliptic, $n$-dimensional manifold admitting a slice-maximal action by a torus $T_M$ of rank $k$.   If $s$ denotes the maximal dimension of an isotropy subgroup, the action being slice maximal is equivalent to the identity $n = k+s$.

Under these hypotheses, there is a torus $K_M \In T_M$ acting almost freely on $M$, with $\dim K_M=k-s$. Since the action of $K_M$ on $M$ is almost free, the orbit space $M/K_M$ is an orbifold $\orb$. Moreover, $\orb$ is rationally elliptic and has an induced action of the torus $T_\orb=T_M/K_M$ of rank $s=\frac{1}{2}\dim\orb$. The images of the $T_M$-orbits of (minimal) dimension $k-s$ under the quotient map $M\to \orb$ correspond to fixed points of the $T_\orb$ action. Hence, $(\orb, T_\orb)$ is a simply-connected, rationally-elliptic torus orbifold. 

By Theorem \ref{T:main_thm}, $\orb$ is $T_\orb$-equivariantly rationally homotopy equivalent to a simply-connected torus orbifold $(\horb=\hP/\hL, T_\orb)$, where $\hP$ is a product of spheres of dimension $\geq 3$ and $\hat{L}$ is a compact abelian Lie group acting linearly and almost freely on $\hat{P}$. Recall from Section~\ref{S:PROOF} that $\hL$ is defined as the kernel of the label map $\lambda: T_{Q}\to T_{\orb}$, where $T_Q=\prod_F S^1_F$ is the product of a copy of $S^1$ for each facet of the orbit space $Q=\orb/T_{\orb} = M/T_M$. Since $\lambda$ is onto, this yields an isomorphism $T_\orb = T_Q/\hL$.

Consider the map $\pi: M\to M/K_M=\orb$, $p\in M$ and $p^*=\pi(p)\in \orb$. A \((T_{\orb})_{p^*}\)-invariant good local chart around $p^*$ is given by $\tilde{U}_{p^*}=\nu_p(K_M(p))$ with map $\tilde{U}_{p^*}\to \orb$ given by the composition $\nu_p(K_M(p))\stackrel{\exp}{\longrightarrow} M\stackrel{\pi}{\longrightarrow} \orb$. The local group at \(p^*\) is given by \(\Gamma_{p^*}=K_{M}\cap (T_{M})_{p}\).  Thus, following the notation of Proposition \ref{P:LES}, one has \((\tilde{T}_{\orb})_{p^*}\subset T_M\) and
\[
(\tilde{T}_{\orb})^o_{p^*}=(T_{M})^o_{p}.
\]

In particular, the slice representation of $(T_{M})^o_{p}$ on $\nu_p(T_M(p))$ coincides with the slice representation of $(\tilde{T}_\orb)_{p^*}$ on $\tilde{\nu}_{p^*}(T_\orb(p^*))$ as in the Slice Theorem (Theorem \ref{T:slice-thm}). From Proposition \ref{P:right-codimension}(b), this action is a sum of a trivial summand and a maximal-rank summand. Such actions belong to a class called \emph{polar actions} and, since every slice representation of $T_M$ is polar, the $T_M$ action on $M$ is \emph{infinitesimally polar} (see \cite{PT,GZ}). 

By definition of the label map, each $\lambda_F:S^1_F\to (T_{\orb})^o_{q^*}\In T_{\orb}$, with $q^*\in \orb$ a point projecting to $F$, factors through $(\lambda_M)_F:S^1_F\to (\tilde{T}_{\orb})^o_{q^*}=(T_M)^o_q$ and, therefore, the map $\lambda:T_Q\to T_\orb$ naturally admits a lift to a map $\lambda_M:T_Q\to T_M$.


\begin{lem}
Let $\lambda_M:T_Q\to T_M$ be the above-defined lift of the label map. Then:
\begin{enumerate}
\item $\lambda_M$ is surjective.
\item For every $p\in M$ projecting to $q\in Q$, the torus $T(q)\In T_Q$ is mapped isomorphically onto $(T_M)^o_p$.
\end{enumerate}
\end{lem}


\begin{proof}
Part (a). Let $M_{\reg}$ denote the collection of principal orbits, and $Q_{\textrm{reg}}=M_{\textrm{reg}}/T_M$. Since the $T_M$ action on $M_\reg$ is free, there is a principal bundle
\[
T_M\to M_\reg\to Q_\reg.
\]
Since $M$ is simply connected and the $T_M$-action on $M$ is infinitesimally polar, by Theorem 1.8 of \cite{Lyt} there are no orbits with finite isotropy and, therefore, the set $Q_\reg$ consists precisely of the orbits of maximal dimension. On the other hand, $Q_\reg$ is the quotient $\orb_\reg/T_\orb$, where $\orb_\reg$ also consists of the orbits of $\orb$ of maximal dimension. Because $\orb$ is a rationally elliptic torus orbifold, $H^{\odd}(\orb;\QQ)=0$. Therefore, Corollary 1 of \cite{bred_free} can be applied to conclude that $Q_\reg$ is rationally acyclic. Since $\pi_1(Q_\reg,[p_0])=\pi_1(Q,[p_0])=0$, by Hurewicz $\pi_2(Q_\reg,[p_0])\otimes \QQ=H_2(Q_\reg;\QQ)=0$ and, in particular, $\pi_2(Q_\reg,[p_0])$ is torsion. From the long exact sequence in homotopy for $M_\reg\to Q_\reg$, it follows that the kernel of $\pi_1(T_M)\to \pi_1(M_\reg,p_0)$ must be torsion as well, but since $\pi_1(T_M)$ is free abelian, the kernel must be trivial. Therefore, the map $\pi_1(T_M)\to \pi_1(M_\reg,p_0)$ is injective.

In order to prove that $\lambda_M$ is surjective, it is enough to show that it induces a surjective map $(\lambda_M)_*:\pi_1(T_Q)\to \pi_1(T_M)$. Letting $\Omega\In \pi_1(T_M)$ denote the image of $(\lambda_M)_*$, from the discussion above it is enough to prove that $\Omega$ has the same image as $\pi_1(T_M)$ in $\pi_1(M_{\reg}, p_0)$.

For any $\alpha \in \pi_1(T_M)$, its image in $\pi_1(M_\reg,p_0)$ is represented by some loop $C$ in a principal $T_M$-orbit in $M$ and, since $M$ is simply connected, hence bounds a two-dimensional disk $D$ in $M$.  The pre-images of the facets of $Q$ are codimension-$2$ submanifolds of $M$.  Hence, by performing a suitable deformation, it may be assumed without loss of generality that $D$ intersects only finitely many of these codimension-$2$ submanifolds, in only finitely many points $x_1, \dots, x_N$, and that these intersections are transversal.  As $D$ is simply connected, $C$ is homotopy equivalent (within $D$) to a concatenation of lassos based at $p_0 \in C$, each of which has a noose $\gamma_i$ which is a circle around a single intersection point $x_i$, $i \in \{1, \dots, N\}$.

For each $i \in \{1, \dots, N\}$, in a sufficiently small neighbourhood of the intersection point $x_i$, the disk $D$ can be assumed to coincide with the normal slice to the $T_M$-orbit through $x_i$.  By the Slice Theorem, a noose $\gamma_i$ around $x_i$ can be assumed to lie in an orbit of the slice action of the one-dimensional isotropy subgroup $(T_M)_{x_i}$, hence, to be some (positive or negative) iterate of the circle $(T_M)_{x_i}^o$.

Together with the isomorphisms arising via change of base points, the above discussion ensures that $C$ is homotopic to the concatenation of the $\gamma_i$, each of which represents an element in $\pi_1(M_\reg,p_0)$ in the image of $\Omega$.
\\

\noindent Part (b). This follows closely the last part of the proof of Theorem A. Given any $p \in M$ with $\dim((T_M)_p) = l$, Proposition \ref{P:Bullet_point_3} states that the image $p^*\in \orb$ of $p$ belongs to the closures $\ol{\Sigma}_i$ of codimension-2 strata $\Sigma_i$, $i=1,\ldots, l$. By Proposition \ref{P:Bullet_point_3} again, each $\ol{\Sigma}_i$ projects to a different facet $F_i$ of $Q$, and it is fixed by a different factor $S^1_i$ of $(T_\orb)_{p^*}^o=T^l$, which lifts to a factor $\tilde{S}^1_i$ of $(\tilde{T}_{\orb})_{p^*}^o=(T_M)_p^o$.

By definition, $(\lambda_M)_{F_i}:S^1_{F_i}\to T_M$ sends $S^1_{F_i}$ isomorphically into $\tilde{S}^1_i\In (T_M)_p^o$ and, therefore, the label map $\lambda$ sends $T([p])=\prod_{i}S^1_{F_i}$ isomorphically into $(T_M)_p^o$, where $[p]\in Q$ is the image of $p$. 
\end{proof}

Let $K_{\hP}\In T_Q$ denote the kernel of $\lambda_M$ and let $\hM$ be the quotient $\hP / K_{\hP}$.  The group $K_{\hM}=\hL/K_{\hP}$ acts almost freely on $\hat{M}$, with quotient $\horb=\hP/\hL$.
Recall, furthermore, that there is an isomorphism 
\[
\vphi : \mathcal{P}(Q)\to \mathcal{P}(\hat{Q}),
\]
of face posets of the quotients $Q=\orb/T_\orb = M/T_M$ and $\hat{Q}=\horb/T_{\orb} = \hP/T_Q$, such that, for every face $F$ of $\mc{P}(Q)$, $F$ and $\vphi(F)$ have the same isotropy.

All of the above information is contained in the following diagram, where the label on each arrow denotes the quotient by the given torus, the dashed line indicates rational homotopy equivalence (not a map!), and the dotted line indicates that there is an isomorphism of face posets $\vphi : \mathcal{P}(Q)\to \mathcal{P}(\hat{Q})$.
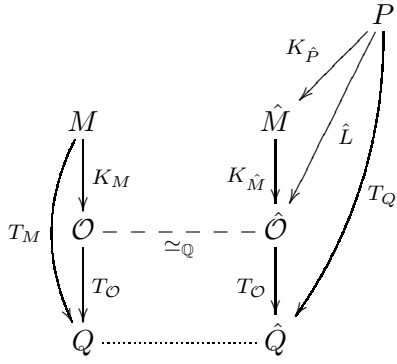
\[
\xymatrix{
& && \hP \ar[ddl]^{\hL} \ar[dl]_{K_{\hP}} \ar@/^1pc/[dddl]^{T_Q}\\
M \ar[d]^{K_M} \ar@/_1pc/[dd]_{T_M} && \hM \ar[d]_{K_{\hM}} & \\
\orb \ar@{--}[rr]_{\simeq_\Q} \ar[d]^{T_\orb} & & \horb \ar[d]_{T_\orb} & \\
Q \ar@{.}[rr] &&  \hat Q &
}
\]
It remains to show that the space $\hM$ is a manifold and that $M$ is $T_M$-equivariantly rationally homotopy equivalent to $\hM$ equipped with the induced action of the torus $T_{\hM} = T_Q/K_{\hP}$.



\begin{lem}
\label{L:FREE_SUBTORUS}
The group $K_{\hP}$ acts freely on $\hP$ and, hence, the quotient $\hM=\hP/K_{\hP}$ is a (topological) manifold.
\end{lem}


\begin{proof}
Recall that $\hat{P}$ is defined by $\hat{Q}\times T_Q/\sim$, where $(q,t)\sim (q',t')$ if and only if $q=q'$ and $t{t'}^{-1}\in T(q)=\prod_{F \ni q} S^1_F \In T_Q$. The action of $T_Q$ on $\hat{P}$ is given by left multiplication on the second factor. The action of $K_{\hat{P}}$ on $\hat{P}$ is simply the restriction of the $T_Q$ action to $K_{\hP}$ and, therefore, the isotropy subgroup of the $K_{\hP}$ action at a point $[(q,t)]\in \hat{P}$ is given by $T(q) \cap K_{\hat{P}}$.

Let $\hat{F}_q$ denote the face of $\hat{Q}$ of minimal dimension containing $q$, and $F_q=\vphi^{-1}(\hat{F}_q)$ the corresponding face in $Q$, given by the face-poset isomorphism $\vphi : \mathcal{P}(Q)\to \mathcal{P}(\hat{Q})$. Since $M$ is a manifold, $T(q)$ maps injectively via $\lambda_M$ into the isotropy of $T_M$ at a point $x\in M$ in the pre-image of $F_q$. Thus $T(q) \cap \ker(\lambda_M) = T(q) \cap K_{\hP}$ is trivial, as desired.
\end{proof}


\begin{lem}
  The manifolds \(\hP/K_{\hP}\) and \(\hP/K_{\hP}^o\) are rationally
  homotopy equivalent. Therefore, in the following it may be assumed that
  $K_{\hP}$ is connected and \(\hat{M}\) is simply connected. 
\end{lem}


\begin{proof}
  It will be shown that the orbit map of the $\Gamma=K_{\hP}/K_{\hP}^o$-action on
  \(\hP/K_{\hP}^o\) induces a rational homotopy equivalence \(\hP/K_{\hP}^o\rightarrow\hP/K_{\hP}\).
  The \(\Gamma\) action commutes with the $K_{\hM}$-action on
  \(\hP/K_{\hP}^o\) and, therefore, induces a $\Gamma$ action on the orbifold
  \(\hat{\orb}'=(\hP/K_{\hP}^o)/K_{\hM}\) with orbit space \(\hat{\orb}=(\hP/K_{\hP})/K_{\hM}\).
  Moreover, there is a commutative diagram
\[
\xymatrix{
K_{\hM} \ar[r]\ar[d]&K_{\hM}/(\Gamma\cap K_{\hM}) \ar[d]\\
\hP/K_{\hP}^o\ar[r]\ar[d]&\hP/K_{\hP} \ar[d]\\
\hat{\orb}'\ar[r]&\hat{\orb}
}
\]
Here the top and bottom maps are rational homotopy equivalences, since
the \(\Gamma\)-actions on \(K_{\hM}\) and \(\hat{\orb}'\) induce trivial actions on cohomology
and the spaces in the corners of the diagram are formal.
Because a model for the spaces in the middle is given by the tensor
product of the models for the corresponding top and bottom spaces, it
follows that the map in the middle is a rational homotopy equivalence.
Hence, it may be assumed that \(K_{\hP}\) is connected.
\end{proof}


Observe now that, since the torus $K_{\hM}={\hL}/K_{\hP}$ acts almost freely on $\hM$ with $\hM/K_{\hM}= \hP/\hL = \horb$,  
the projection $\hat{M}\to \hat{\orb}$ is, up to rational homotopy, a principal $K_{\hat{M}}$-bundle.

The label map $\lambda_M:T_Q\rightarrow T_M$ described above descends to an isomorphism $\ul{\lambda}_M:T_{\hat{M}}\to T_M$ with inverse $\mu_M: T_M\to T_{\hat{M}}$. 
Since  $\lambda: T_Q\to T_\orb$ factors through $\lambda_M$, there is an induced map $\hat{\pi}:T_{\hat{M}}\to T_\orb$ with kernel $K_{\hat{M}}$. Then the following diagram commutes
\begin{equation}
\label{E:Tori}
\xymatrix{
T_M\ar[r]^{\pi} \ar[d]_{\mu_M} &  T_\orb \ar[d]^{=} \\
T_{\hat{M}} \ar[r]^{\hat{\pi}} &  T_\orb
}
\end{equation}
where the vertical maps are isomorphisms. Moreover, there is an induced isomorphism $\mu_K:K_M\to K_{\hat{M}}$ given by the restriction of $\mu_M$ to $K_{M}$.   Therefore, $M\to \orb$ and $\hat{M}\to \hat{\orb}$ can be thought of as rational homotopy principal $K_M$-bundles, and the goal is to show that $M$ and $\hat{M}$ are rationally homotopy equivalent.


\begin{thm}
\label{T:BUNDLE_RHEQ}
Let $X$, $Y$ be rationally homotopy equivalent spaces, and let $\phi:H^2(Y;\QQ)\to H^2(X;\QQ)$ be the isomorphism induced by a rational equivalence. Moreover, let $T$ be a $k$-torus and let $\xi_X:E_X\to X$, $\xi_Y:E_Y\to Y$ be rational homotopy principal $T$-bundles with classifying maps $\rho_X:X\to BT$, $\rho_Y:Y\to BT$.
Suppose, finally, that there is a map $\beta:H^2(BT;\Q)\to H^2(BT;\Q)$ such that the diagram
\begin{equation}
\label{E:DIAG_BUNDLE_RHEQ}
\xymatrix{
H^2(X;\QQ)& H^2(Y;\QQ) \ar[l]_{\phi}\\
H^2(BT;\QQ) \ar[u]^{\rho_X^*} & H^2(BT;\QQ) \ar[u]^{\rho_Y^*} \ar[l]_{\beta}
}
\end{equation}
commutes. Then $E_X$ is rationally homotopy equivalent to $E_Y$.
\end{thm}


\begin{proof}
Let $(\wedge V_X, d_X)$ and $(\wedge V_Y,d_Y)$ be the minimal models of $X$ and $Y$ respectively. Let $\varphi:(\wedge V_Y,d_Y)\to (\wedge V_X,d_X)$ be an isomorphism inducing $\phi:H^2(Y;\Q)\to H^2(X;\Q)$. 

The minimal model of $T$ is $(\wedge(t_1,\ldots,t_k),0)$ with $|t_i|=1$.  The minimal model of $BT$ is $\Q[\bar{t}_1,\ldots,\bar{t}_k]$, where $|\bar{t}_i|=2$. The  $t_i$ are mapped to $\bar{t}_i$ via the isomorphism
\[
\delta : W = \Hom(\pi_1(T),\Q) \to \ol W = \Hom(\pi_2(BT), \Q)
\] 
induced by the long exact homotopy sequence of the  fibration $T\to ET\to BT$.  It's clear that $H^2(BT;\Q)$ can now be identified with the vector space $\ol{W} = \Span_\Q\{\ol{t}_1,\ldots,\ol{t}_k\}$.  Using $\delta$, the map $\beta:\ol{W}\to\ol{W}$ induces a map 
\[
\check\beta = \delta^{-1} \circ \beta \circ \delta :W\to W.
\]

 A model for $E_X$ is $(\wedge V_X\otimes\wedge(t_1,\ldots,t_k),D_X)$, where $D_X |_{\wedge V_X}=d_X$ and $D_X|_{W}=\rho^*_X\circ \delta$. Similarly, a model for $E_Y$ is $(\wedge V_Y\otimes\wedge(t_1,\ldots,t_k),D_Y)$, where $D_Y |_{\wedge V_Y}=d_Y$ and $D_Y|_{W}=\rho^*_Y\circ \delta$. Define now an isomorphism
 \[
 \psi:(\wedge V_Y\otimes\wedge(t_1,\ldots,t_k),D_Y)\longrightarrow (\wedge V_X\otimes\wedge(t_1,\ldots,t_k),D_X)
\]
by letting $\psi|_{\wedge V_Y}=\varphi$ and $\psi(1\otimes t_i)=1\otimes \check\beta(t_i)$. It is clear that $\psi$ preserves the grading and $(\psi\circ D_Y)|_{\wedge V_Y}=(D_X\circ\psi))|_{\wedge V_Y}$. Moreover, using the commutativity of diagram~\eqref{E:DIAG_BUNDLE_RHEQ} (and Hurewicz),
\begin{align*}
\psi\circ D_Y |_{W} & = \psi\circ \rho^*_Y\circ \delta \\
			  &= \vphi \circ \rho^*_Y\circ \delta \\
			  &= \rho_X^* \circ \beta \circ \delta \\
			  &=\rho^*_X\circ \delta \circ \check\beta\\
			  & = D_X|_{W}\circ \check\beta \\
			  & = D_X\circ\psi|_{W}.
\end{align*}
Then $\psi$ is an isomorphism between the models of $E_X$ and $E_Y$. Consequently, there is an isomorphism between the corresponding minimal models and $E_X \simeq_\Q E_Y$. 
\end{proof}

It is now apparent that, in order to show that $M$ and $\hat{M}$ are rationally homotopy equivalent, it suffices to show that the hypotheses of Theorem \ref{T:BUNDLE_RHEQ} are satisfied by the rational homotopy principal $K_M$-bundles $M\to \orb$ and $\hat{M}\to \hat{\orb}$.



\begin{prop}
\label{P:COMM_DIAG}
The  diagram
\begin{equation}
\label{E:CD_BUNDLES}
\xymatrix{
H^2(\orb;\QQ)& H^2(\hat{\orb};\QQ) \ar[l]_{f}\\
H^2(BK_M;\QQ) \ar[u]^{} & H^2(BK_{\hat{M}};\QQ) \ar[u]^{}\ar[l]_{(B\mu_K)^*}
}
\end{equation}
is commutative, where the vertical arrows are induced by the bundles $M\to \orb$ and $\hat{M}\to \hat{\orb}$, and $f$ is the isomorphism given by $\orb \simeq_\Q \horb$.
\end{prop}

As a first step towards proving Proposition \ref{P:COMM_DIAG}, the following lemma is necessary.


\begin{lem}
\label{L:1-skeleta}
  Let \(M^{(1)}\subset M\), \(\hat{M}^{(1)}\subset \hat{M}\) be the pre-images of the vertex-edge graphs of $Q$ and $\hat{Q}$, respectively.
  Then there is a \(T_M\)-equivariant homeomorphism \(\tilde h : M^{(1)}\rightarrow \hat{M}^{(1)}\).

  Moreover, $\tilde h$ induces a $T_{\orb}$-equivariant homeomorphism $h:\orb^{(1)}\to \hat{\orb}^{(1)}$ whose induced map in cohomology  completes the commutative diagram
\[
\xymatrix{
H^2(\orb;\QQ) \ar[d]_{i^*}  & H^2(\hat{\orb};\QQ) \ar[l]_{f} \ar[d]^{\hat i^*}  \\
H^2(\orb^{(1)};\QQ) & H^2(\horb^{(1)};\QQ) \ar[l]_{h^*}
}
\]
where the vertical maps are induced by the respective inclusions.
\end{lem}


\begin{proof}
  \(M^{(1)}\) and \(\hat{M}^{(1)}\) are unions of cohomogeneity-one manifolds \(N_{ij}\) and \(\hat{N}_{ij}\), respectively.
  Each of these cohomogeneity-one manifolds is the pre-image of an edge in \(Q\) or \(\hat{Q}\), respectively.
Moreover, the indices run over the edges $e_{ij}$ of the vertex-edge graph of $Q$. 

  Since the isotropy subgroups of the \(T_M\)-action on each \(N_{ij}\) and \(\hat{N}_{ij}\) are the same, there are equivariant homeomorphisms \(N_{ij}\rightarrow \hat{N}_{ij}\).
  Since \(T_M\) is a compact, connected, abelian Lie group, these homeomorphisms can be chosen in such a way that they extend to an equivariant homeomorphism \(\tilde h: M^{(1)}\rightarrow \hat{M}^{(1)}\).

  Because \(\orb^{(1)}=M^{(1)}/K_M\) and \(\hat{\orb}^{(1)}=\hat{M}^{(1)}/K_M\), it follows that there is an \(T_\orb\)-equivariant homeomorphism \(h : \orb^{(1)}\rightarrow \hat{\orb}^{(1)}\).

It remains to show that the induced map in cohomology  completes a commutative diagram as in the statement of the lemma. Consider the diagram below, where the back (by equivariant rational homotopy equivalence), base and sides are each commutative.  The goal is to show that the dotted arrow in the diagram below makes the top of the cube into a commutative diagram.
\[
\xymatrix{
H^2(\orb;\QQ) \ar[drr]_(0.3){i^*}  & H^2(\hat{\orb};\QQ) \ar[l]_{f} \ar[drr]^{\hat i^*}  & & \\
& & H^2(\orb^{(1)};\QQ) & H^2(\horb^{(1)};\QQ) \ar@{.>}[l]^{h^*}\\
H^2_{T_\orb}(\orb;\QQ) \ar[drr]_{i^*} \ar[uu] & H^2_{T_\orb}(\hat{\orb};\QQ) \ar[l] \ar[drr]^(0.3){\hat i^*} |!{[r];[dr]}\hole  \ar[uu] |!{[uul];[ur]}\hole& & \\
& & H^2_{T_\orb}(\orb^{(1)};\QQ) \ar[uu] & H^2_{T_\orb}(\horb^{(1)};\QQ) \ar[l] \ar[uu]\\
}
\]
Here the bottom maps are induced by functoriality from the isomorphism of face posets $\varphi:\mathcal{P}(Q)\to \mathcal{P}(\hat{Q})$, see Remark \ref{R:functorial}.

By a diagram chase, one readily sees that it suffices to show both that the vertical map $H^2_{T_\orb}(\orb; \QQ) \to H^2(\orb; \QQ)$ is surjective and that the front of the cube is commutative.

Since $H^1(\orb,\Q)=0$ the natural map $H^2_{T_\orb}(\orb; \QQ) \to
H^2(\orb; \QQ)$ is surjective.

By using an inductive Mayer-Vietoris sequence argument, one sees that $H^2(\orb^{(1)};\QQ)$ is generated by the duals $\alpha_{ij}$ of the fundamental classes $[\sph^2_{ij}]\in H_2(\orb^{(1)};\QQ)$. Similarly, $H^{2}(\hat{\orb}^{(1)};\QQ)$ is generated by the duals $\hat{\alpha}_{ij}$ of the classes $[\hat{\sph}^2_{ij}]\in H_{2}(\hat{\orb}^{(1)};\QQ)$, and $h^*(\hat{\alpha}_{ij})=\alpha_{ij}$.
	
On the other hand, by using the Mayer-Vietoris sequence on the Borel construction $(\sph^2_{ij})_T$ (using the decomposition $(\sph^2_{ij})_T = U \cup V$, with $U=(\sph^2_{ij}\setminus p_i)_T$ and $V=(\sph^2_{ij}\setminus p_j)_T$)
\[
H^2_T(\sph^2_{ij};\ZZ)=\{(g_i,g_j)\in H^2_T(p_i;\ZZ)\oplus H^2_T(p_j;\ZZ)\mid g_i-g_j\in  \ZZ\cdot \mu(e_{ij})\}
\]
where $\mu(e_{ij})$ is the weight of $e_{ij}$ in the GKM-graph associated to $\orb$. From the Serre spectral sequence of the fibration $\sph^2_{ij}\to (\sph^2_{ij})_T\to BT$, which degenerates in the $E_2$-page, it follows that the map $H^2_T(\sph^2_{ij};\ZZ)\to H^2(\sph^2_{ij};\ZZ)$ is surjective, and sends $(\mu(e_{ij}),0)$ to $\alpha_{ij}$. The same discussion carries over identically to the spheres $\hat{\sph}^2_{ij}$. The map $h^*:H^2(\hat{\orb}^{(1)}; \QQ)\to H^2(\orb^{(1)}; \QQ)$ can now be factored as
\begin{align*}
H^2(\hat{\orb}^{(1)};\QQ)&\lra H^2_T(\hat{\orb}^{(1)};\QQ)  &\lra && H^2_T(\orb^{(1)};\QQ)&\lra H^2(\orb^{(1)};\QQ)\\
\hat{\alpha}_{ij}&\lmt (\hat{\mu}(\hat{e}_{ij}),0) &\lmt && (\mu(e_{ij}),0)&\lmt \alpha_{ij}
\end{align*}
and, therefore, there is a commutative diagram, as desired.
\end{proof}


\begin{lem}
\label{L:INJ_I}
The inclusion map $i:\orb^{(1)}\to \orb$ induces an injection $i^*:H^2(\orb;\Q)\to H^2(\orb^{(1)};\Q)$.
\end{lem}


\begin{proof}
Recall that there is a map of fibrations
\begin{equation}
\xymatrix{
\orb^{(1)}\ar[d]_{i} \ar[r]^{}& \orb^{(1)}_{T_\orb} \ar[r]^{} \ar[d]_{i^*} & BT_\orb \ar@{=}[d]^{}\\
{\orb} \ar[r]^{} & {\orb}_{T_\orb} \ar[r]^{}& BT_\orb 
}
\end{equation}
which induces a map between the corresponding Serre spectral sequences with rational coefficients. Both spectral sequences have the property that $E_2^{1,j}=E_2^{3,j}=0$ for all $j\geq 0$.  Therefore, for $X = \orb^{(1)}$ or $X=\orb$ there are exact sequences
\begin{equation*}
  0\rightarrow E_{\infty}^{2,0}\rightarrow H^2(X_{T_\orb};\QQ)\rightarrow E_\infty^{0,2}\rightarrow 0,
\end{equation*}
The natural map \(H^2(BT_\orb;\QQ)\rightarrow H^2(X_{T_\orb};\QQ)\) is injective, since there are \(T_\orb\)-fixed points in \(X\). It then follows that \(E_\infty^{2,0}=H^2(BT_\orb;\QQ)\).
Moreover, \(E_\infty^{0,2}\In H^2(X;\QQ)\), with equality holding if \(b_1(X)=0\). This last condition holds if \(X=\orb\)

The map $i^*$, in particular, induces a row-exact, commutative diagram
\[
\xymatrix{
 0 \ar[r] 	&  H^2(BT_\orb;\QQ) \ar[r] 		& H^2(\orb^{(1)}_{T_\orb};\QQ) \ar[r] 		& H^2(\orb^{(1)};\Q)&\\
 0 \ar[r]				&  H^2(BT_\orb;\QQ) \ar[r] \ar@{=}[u]	& H^2(\orb_{T_\orb};\QQ) \ar[r]\ar@^{(->}[u]^{(i_{T_\orb})^*} 	& H^2(\orb;\Q) \ar[r] \ar[u]^{i^*}&0
}
\]
By \cite[Proof of Prop.\ 2.4]{ChSk}, $(i_{T_\orb})^*$ is injective and, by diagram chasing it follows that  $i^*:H^2(\orb;\QQ)\to H^2(\orb^{(1)};\QQ)$ is injective as well.
\end{proof}


\begin{proof}[Proof of Proposition~\ref{P:COMM_DIAG}] Diagram~\eqref{E:CD_BUNDLES} is part of the larger diagram

\begin{equation}
\label{E:3-squares-diagram}
\xymatrix{
H^2(\orb^{(1)};\QQ)& H^2(\hat{\orb}^{(1)};\QQ) \ar[l]_{h^*}\\
H^2(\orb;\QQ) \ar[u]^{i^*}& H^2(\hat{\orb};\QQ) \ar[u]^{\hat{i}^*} \ar[l]_{f}\\
H^2(BK_M;\QQ) \ar[u]^{} & H^2(BK_{\hat{M}};\QQ) \ar[u]^{}\ar[l]_{(B\mu_K)^*}
},
\end{equation}
where, $i:\orb^{(1)}\to \orb$, $\hat{i}:\hat{\orb}^{(1)}\to \hat{\orb}$ denote the inclusions. The upper square comes from Lemma~\ref{L:1-skeleta} and hence commutes. By Lemma~\ref{L:INJ_I}, the map $i^*$ is injective.  In order to prove the proposition it now suffices to show that the outer square commutes.

By Lemma~\ref{L:1-skeleta}, the map $h:\orb^{(1)}\to \hat{\orb}^{(1)}$ can be lifted to a $\mu_K$-equivariant map $\tilde{h}:M^{(1)}\to \hat{M}^{(1)}$ 
such that the diagram 
\begin{equation*}
\xymatrix{
K_M 	    \ar[d]	 \ar[r]^{\mu_K}	&  	K_{\hM} 	\ar[d]\\
M^{(1)} \ar[d]	 \ar[r]^{\tilde{h}}		&	 \hM^{(1)}	\ar[d]\\
\orb^{(1)} 		\ar[r]^{h}	&	 \horb^{(1)}}
\end{equation*}
is a pull-back diagram between the (rational homotopy) principal torus bundles  $M^{(1)}\to \orb^{(1)}$ and  $ \hat{M}^{(1)}\to \hat{\orb}^{(1)}$. This induces a (rational homotopy) commutative diagram 
\begin{equation*}
\xymatrix{
\orb^{(1)} \ar[d]	 \ar[r]^{h}	&	\horb^{(1)}	\ar[d]\\
BK_M  		\ar[r]^{B\mu_K}			&	BK_{\hM}  }
\end{equation*}
from which the commutativity of the outer square in diagram \eqref{E:3-squares-diagram} follows. 
\end{proof}


\begin{thm} The manifolds $M$ and $\hM$ are $T_M$-equivariantly rationally homotopy equivalent.
\end{thm}


\begin{proof}
Recall that $M\to \orb$ (resp.\ $\hat{M}\to \hat{\orb}$) is a rational principal $K_M$-bundle (resp.\ $K_{\hat{M}}$-bundle), with an isomorphism $\mu_K : K_M  \to K_{\hat M}$.  With respect to the induced identification $(B\mu_K)^* : H^*(BK_{\hat M}; \Q) \to H^*(BK_M; \Q)$, Proposition \ref{P:COMM_DIAG} yields a commutative diagram 
$$
\xymatrix{
H^2(\orb;\QQ)& H^2(\hat{\orb};\QQ) \ar[l]_{f}\\
H^2(BK_M;\QQ) \ar[u] & H^2(BK_M;\QQ) \ar[u] \ar@{=}[l]
}
$$
where $f$ is the isomorphism induced by the rational homotopy equivalence $\orb \simeq_\Q \hat \orb$ and the vertical arrows are induced by the rational principal $K_M$-bundles.  By Theorem \ref{T:BUNDLE_RHEQ}, this implies that the total spaces $M$ and $\hat M$ are rationally homotopy equivalent.

Consider now the diagram
$$
\xymatrix{
H^*(M; \Q) & H^*(\hat M; \Q) \ar[l] \\
H^*(\orb;\Q) \ar[u] & H^*(\hat{\orb};\Q) \ar[u] \ar[l]_{f}\\
H_{T_\orb}^*(\orb;\Q) \ar[u] & H_{T_\orb}^*(\hat \orb;\Q) \ar[u] \ar[l]
}
$$
where the uppermost map is the isomorphism induced by the rational
homotopy equivalence $M \simeq_\Q \hat M$ constructed in Theorem
\ref{T:BUNDLE_RHEQ}.  From that construction, it is clear that the
upper square commutes.  On the other hand, the lower square commutes
because of the equivariance of the rational homotopy equivalence $\orb
\simeq_\Q \hat \orb$.  Since $\orb/T_\orb = M/T_M$ and $\hat
\orb/T_\orb = \hat M / T_{\hat M}$, it follows from the commutativity
of \eqref{E:Tori} that $M$ and $\hat M$ are $T_M$-equivariantly
rationally homotopy equivalent.

Indeed, since \(K_M\) acts almost freely on \(M\) with orbit space
\(\orb\), there is an isomorphism
\begin{equation*}
  H^*_{T_M}(M;\Q)\cong H^*_{T_\orb}(\orb;\Q).
\end{equation*}
and similarly for \(\hat{M}\) and \(\hat{O}\).
Therefore, since the above diagram commutes, there is a commutative diagram
$$
\xymatrix{
H^*(M; \Q) & H^*(\hat M; \Q) \ar[l] \\
H_{T_M}^*(M;\Q) \ar[u] & H_{T_M}^*(\hat{M};\Q) \ar[u] \ar[l]
}
$$
as desired.
\end{proof}

\section{Torus actions in non-negative curvature}
\label{S:Nonneg}

To begin this section, a version of Theorem~\ref{T:main_thm} for non-negatively-curved torus orbifolds of dimension at most six will be established. 


\begin{thm}
\label{T:6mfds}
Let $(\orb,G)$ be a non-negatively-curved and simply-connected torus orbifold of dimension at most six such that \(H^{\text{odd}}(\orb;\QQ)=0\).
Then there is a product $\hP$ of spheres of dimension $\geq 3$, a torus $\hat{L}$ acting linearly and almost freely on $\hP$, a linear action of $G$ on  $\hat{\orb}=\hP/\hat{L}$ and a $G$-equivariant rational homotopy equivalence $\orb\simeq_{\QQ}\hat{\orb}$.
\end{thm}

To prove this theorem, it suffices to show that \(\orb/G\) satisfies all the properties listed  in Proposition~\ref{P:comb-properties}.  All of these, except for Property (f), can be proved as in the rationally-elliptic case.

Note that, as \(\orb/G\) is being viewed as a face of itself of codimension zero, in order to prove Theorem \ref{T:6mfds}, Property (f) needs to be discussed in dimension six.  Since the rational cohomology of \(\orb\) is concentrated in even degrees, it follows that all faces of \(\orb/G\) are acyclic over the rationals \cite[Corollary 3]{bred_free}. Hence, the following lemma implies that Property (f) holds for \(\orb/G\).


\begin{lemma}
  There is no simply-connected, six-dimensional torus orbifold \((\orb,G)\) such that each face of \(\orb/G\) is acyclic over the rationals, each facet of \(\orb/G\) is combinatorially equivalent to a square, and the intersection of any two facets has two components.
\end{lemma}
\begin{proof}
  Assume that there is a torus orbifold whose orbit space contradicts the conclusion of the lemma.  
  First note that all two-dimensional orbifolds are homeomorphic to two-dimensional topological manifolds.
  Therefore, since the facets of \(\orb/G\) are acyclic over the rationals and orientable, they are all homeomorphic to two-dimensional discs.
  Hence, with the same argument as in the proof of Lemma 4.4 of \cite{Wi}, one sees that the boundary of \(\orb/G\) is homeomorphic to \(\mathbb{R} P^2\).  However, \(\orb/G\) is an orientable orbifold with boundary, while \(\mathbb{R} P^2\) is non-orientable, yielding a contradiction, as desired.
\end{proof}


\begin{proof}[Proof of Theorem \ref{T:6mfds}]
Since $H^\odd(\orb;\Q) = 0$ and $\orb$ admits an invariant metric with non-negative sectional curvature, the conclusion of
Proposition~\ref{P:comb-properties} holds for \(\orb/G\) as discussed above.
Therefore the same arguments as in the proof of Theorem~\ref{T:main_thm} can be carried out to prove Theorem \ref{T:6mfds}.
\end{proof}

To conclude the article, a proof of Theorem \ref{T:Bott} is provided, that is, the Bott Conjecture in the presence of an isometric, slice-maximal torus action is verified.  This is a generalisation of Theorem 1.2 of \cite{Sp}.


\begin{proof}[Proof of Theorem \ref{T:Bott}]
Let $T$ denote the torus whose action on $M$ is slice maximal.  It is sufficient to show that $M$ is \emph{rationally $\Omega$-elliptic}, i.e. that the pointed loop space $\Omega M$ of $M$ satisfies $\sum_r \dim(\pi_r (\Omega M) \otimes \Q) < \infty$, since, $M$ being simply connected, this property implies that $M$ is rationally elliptic. The proof will proceed by induction on the dimension $d = \dim(M/T)$ and no longer assumes that $M$ is simply connected.

When $d=0$, $M$ consists of one orbit and is, therefore, a torus, hence rationally $\Omega$-elliptic.  Suppose now that every non-negatively-curved, closed manifold admitting an isometric, slice-maximal torus action with quotient of dimension $d-1$ is rationally $\Omega$-elliptic.  

From the introduction, the action of $T$ on $M$ being slice maximal ensures that, at every point on a (fixed) minimal orbit, the normal slice is even dimensional and the identity component $G$ of the isotropy subgroup acts on it with maximal rank, i.e. the action is equivalent to the standard linear, effective action of $G$ on $\CC^{\dim(G)}$.  Hence, one can find a circle subgroup $S \In G \In T$ such that some component $M'$ of its fixed-point set $M^S$ is of codimension two and contains the minimal orbit.  Consequently, the induced action of $T' = T/S$ on $M'$ is slice-maximal.  Moreover, since $M'$ is totally geodesic, hence non-negatively curved, and $\dim(M'/T') = d-1$,  the induction hypothesis yields that $M'$ is rationally $\Omega$-elliptic.

As the action of $S$ on $M$ is fixed-point homogeneous, by Theorem 4.1 of \cite{Sp} there exists a submanifold $N \In M$ such that $M$ is diffeomorphic to the union of the normal disc bundles $D(M')$ and $D(N)$ of $M'$ and $N$ along their common boundary $E$:
\[
M = D(M') \cup_{E} D(N).
\]
The foot-point projection $D(M')\to M'$ induces an $S^1$-bundle $E\to M'$. Since $M'$ is rationally $\Omega$-elliptic, it follows from the homotopy long exact sequence that $E$ is also rationally $\Omega$-elliptic. Moreover, by Theorem D of \cite{GH}, the homotopy fibre $F$ of the inclusion $\iota: E\hookrightarrow M$ is rationally $\Omega$-elliptic. Therefore, from the homotopy long exact sequence for $\iota$ and the fact that $E$ is rationally $\Omega$-elliptic, it follows that $M$ is rationally $\Omega$-elliptic as well, as desired.
\end{proof}


\bibliographystyle{amsplain}


\end{document}